\documentclass[10pt]{article}
\usepackage{amsfonts}
\usepackage{amsmath}
\usepackage{amssymb}
\usepackage{amsthm}
\usepackage{mathrsfs}
\usepackage{graphicx}
\usepackage{tabularx}

\usepackage{algorithmic} 

\newtheorem{theorem}{Theorem}[section]
\newtheorem{example}{Example}[section] 

\newtheorem{algorithm}{Algorithm}[section]

\newtheorem{definition}{Definition}[section]

\newcommand{\iq}{{\bf i}}
\newcommand{\jq}{{\bf j}}
\newcommand{\kq}{{\bf k}}

\newcommand{\Aq}{{\bf A}}

\newcommand{\Vq}{{\bf V}}
\newcommand{\vq}{{\bf v}}
\newcommand{\Uq}{{\bf U}}
\newcommand{\uq}{{\bf u}}

\newcommand{\Xq}{{\bf X}}

\newcommand{\Fq}{{\bf F}}

\title {\Large{\bf  The Multi-Symplectic Lanczos Algorithm and Its Applications to Color Image Processing\thanks{This paper is supported in part by  the National Natural Science Foundation
of China under grants 11771188,  the Priority Academic Program Development Project (PAPD),
 the Top-notch Academic Programs Project  (No. PPZY2015A013)  of Jiangsu Higher Education
Institutions,  and the Natural Science Foundation of the Jiangsu Higher Educations of China (No.18KJA110001)}
}}

\author{Zhigang Jia\thanks{Corresponding author.  E-mail: zhgjia@jsnu.edu.cn},
Xuan Liu, and
Mei-Xiang Zhao,
 \\
 School of Mathematics and Statistics  \&  Jiangsu Key Laboratory \\ of Education Big Data Science and
Engineering, \\ Jiangsu Normal University,
Xuzhou 221116, China
}

\date{}

\begin{document}
\maketitle

\begin{abstract}
Low-rank approximations of original samples are playing more and more an important role in many recently proposed mathematical models from data science.  A natural and initial requirement is that these representations inherit original structures or properties.  With this aim, we propose a new multi-symplectic  method based on the Lanzcos bidiagonalization to compute the partial singular triplets of JRS-symmetric matrices. These singular triplets can be used to reconstruct optimal low-rank approximations while preserving the intrinsic multi-symmetry.  The augmented Ritz and  harmonic Ritz vectors are used to perform implicit restarting to obtain a satisfactory bidiagonal matrix for calculating the $k$ largest  or  smallest singular triplets, respectively.  We also apply the new multi-symplectic Lanczos  algorithms to  color face recognition and color video compressing and reconstruction.  Numerical experiments indicate their superiority over the state-of-the-art algorithms.
\end{abstract}


\section{Introduction}\label{s:1}
 The optimal low-rank approximations are the core targets of many recently proposed color image processing models, such as the two dimensional principle component analysis (2DPCA) for color face recognition \cite{jlz17}, the robust PCA for color image inpainting \cite{jns19nla}, and  color video compressing and reconstruction \cite{jns19na}.  For digital images, we can apply the partial singular value decomposition (SVD) based on the  Lanzcos bidiagonalization to reconstruct such low-rank approximations. If the samples have algebraic structures, then the Lanzcos method is expected to be able to preserve these structures which usually reflect the intrinsic relations among different sceneries in the pictures. In this paper, we propose a new multi-symplectic  Lanczos bidiagonalization method to preserve the algebraic multi-symmetry of matrices (which represent color images); and the augmented Ritz and harmonic Ritz vector restarting techniques are implicitly applied  to obtain the  $k$ largest  or  smallest singular triplets, respectively.  This structured Lanczos method will be applied to  the color image processing and its own unique advantages will be indicated by numerical experiments.

 The Lanczos bidiagonalization method \cite{G.H. Golub2013,H. D. Simon2000} has been widely used to compute the partial singular value decomposition of real or complex matrices of large-scale size. Bj\"{o}rck, Grimme and  Van Dooren firstly proposed the implicitly restarted technique based on Lanczos bidiagonalization decomposition (LBD) in \cite{A. Bjorck1994}.   Jia and Niu  \cite{Z. Jia2003} proposed the implicitly restarted LBD with   ``refined shifts''  to calculate  partial largest or smallest singular triplets. In order to increase the accuracy of the calculated  smallest singular triplets, Kokiopoulou, Bekas, and  Gallopoulos \cite{kbg04} proposed the algorithm \texttt{IRLANB}, in which the harmonic Ritz values are used to effectively approximate the smallest singular values.  Later, Baglama, and Reichel \cite{bare05}  proposed the algorithms \texttt{IRLBA(R)} and \texttt{IRLBA(H)}, by which the  $k$ largest  or  smallest singular triplets of  large-scale matrices are obtained from the Ritz vector or harmonic Ritz vector augmented Krylov subspace. In the recent paper \cite{jns19na}, the LBD method was generalized to quaternion matrices and the algorithm \texttt{LANSVDQ}  was developed to compute the $k$ largest singular triplets and was applied to color face recognition,  color video compressing, etc. To the best of our knowledge, there is still no efficient algorithm  to compute  the $k$ smallest singular triplets of quaternion matrices. The multi-symplectic  Lanczos method proposed in this paper can be applied to compute both  the $k$ largest  singular triplets and the $k$ smallest ones  of  large-scale quaternion matrices.

The structure-preserving idea is to keep the (algebraic) symmetries or properties of  continuous or discrete  equations  in the solving process. The aim is to  accurately, stably and efficiently calculate the required solution.  One of the most famous structured matrices in numerous powerful  algorithms is  the (block) Toeplitz matrix, which is generated in the calculation of differential equations from practical applications, 
see, e.g., \cite{Chan1996,chji07,Michael2004,jin10,kns15}.  Another one is the (skew-)Hamiltonian matrix,  which arises in the control theory and on which  a lot of structure-preserving algorithms and perturbation analysis have been developed \cite{benner01,bbmx02,Peter2011,mms16,mmw18}. In this paper,  we treat an important multi-symmetry,  which was firstly proposed in \cite{jwl13}, and then widely studied in quaternion computation  \cite{jwzc18, jns19na, Li2014, Li2016, wei2018} and in color image processing \cite{jlz17,jns19nla, jnw19,ljly19}. This new algebraic symmetry will be preserved  in the proposed multi-symplectic Lanczos method.
The goal is to combine stability with the speed of actual calculations by performing only real operations.

The rest of this paper is structured as follows. In Section \ref{s:prim}, we introduce a new algebraic symmetry  and propose the multi-symplectic transformations. In Section \ref{s:MSL}, we propose a multi-symplectic Lanczos bidiagonalization method for JRS-symmetric matrices and present the  restarted multi-symplectic Lanczos method  to calculate partial largest or smallest singular triplets of JRS-symmetric matrices. In Section \ref{s:cip}, we apply the proposed algorithms to  compute partial singular triplets of quaternion matrices and  apply it to  color image processing. In Section \ref{s:ne}, we present numerical examples to demonstrate the efficiency of the proposed algorithms on computing partial singular triplets. In Section \ref{s:clu}, we summarize the main work of this paper.

\section{Primaries} \label{s:prim}

In this section, we introduce  JRS-symmetric matrices and 
multi-symplectic transformations.

 \begin{definition}[JRS-symmetry \cite{jwl13,jwzc18}] Define three skew-symmetric matrices as   
\begin{subequations}\label{e:JRS}
\begin{align*}
                            &J_n=\left[
                                \begin{array}{cccc}
                                 0 & 0 & -\emph{I}_n & 0 \\
                                  0 & 0 & 0&-\emph{I}_n  \\
                                  \emph{I}_n& 0 & 0 & 0 \\
                                 0 &\emph{I}_n & 0&  0\\
                                \end{array}
                              \right],~
                              R_n=\left[
                                \begin{array}{cccc}
                                 0&-\emph{I}_n & 0 & 0 \\
                                  \emph{I}_n & 0 & 0& 0 \\
                                  0 & 0 & 0 & \emph{I}_n \\
                                 0 & 0& -\emph{I}_n& 0\\
                                \end{array}
                              \right],\\
                               & S_n=\left[
                                \begin{array}{cccc}
                                 0& 0& 0 & -\emph{I}_n \\
                                 0 & 0 &\emph{I}_n  & 0 \\
                                  0 & -\emph{I}_n & 0 &0 \\
                                \emph{I}_n& 0&0& 0\\
                                \end{array}
                              \right].
                              \end{align*}
                              \end{subequations}
 A non-zero matrix $M\in\mathbb{R}^{4m\times 4n}$ is called JRS-symmetric if
\begin{equation}\label{e:msm}
     J_mMJ_n^T= R_mMR_n^T=  S_mMS_n^T=M.
\end{equation}
A matrix  $O\in\mathbb{R}^{4m\times 4n}$ is called multi-symplectic if
  \begin{equation}\label{e:mso}
    OJ_nO^T=J_m,~ OR_nO^T=R_m,~ OS_nO^T=S_m.
 \end{equation}
\end{definition}

A $JRS$-symmetric matrix $M\in\mathbb{R}^{4m\times 4n}$  has the following algebraic structure,
\begin{equation}\label{e:Mstr}
  M:=\left[
  \begin{array}{cccc}
    M^{(0)} & M^{(2)} & M^{(1)} & M^{(3)} \\
    -M^{(2)} & M^{(0)} & M^{(3)} & -M^{(1)} \\
    -M^{(1)} & -M^{(3)} & M^{(0)} & M^{(2)} \\
    -M^{(3)}& M^{(1)} & -M^{(2)} & M^{(0)} \\
  \end{array}
\right], M^{(i)}\in\mathbb{R}^{m\times n}, i=0,1,2,3.
\end{equation}
In fact, a real matrix is $JRS$-symmetric if and only if it is a real counterpart of a quaternion matrix \cite{jwzc18}.  Based on the quaternion representation,  
a color image with the spatial resolution of $m\times n$ pixels  \cite{jns19na} can be represented by an $m\times n$ pure quaternion matrix $\Aq$ in $\mathbb{Q}^{m\times n}$  as follows:
\begin{equation*}
 \Aq_{ij} = R_{ij}\iq + G_{ij}\jq + B_{ij}\kq,\ \  1 \leq i \leq  m, 1 \leq j \leq n,
\end{equation*}
where $R_{ij}$, $G_{ij}$ and $B_{ij}$ are the red, green and blue pixel values at the location $(i,j)$ in the image, respectively, and  $\iq,\jq,\kq$  are three units satisfying $\iq^2 = \jq^2 = \kq^2 = \iq\jq\kq = -1.$
For instance, the color image in Figure \ref{fig:toy} can be stored in the quaternion matrix
$\Aq=R\iq + G\jq + B\kq$,  and thus can be represented by the  JRS-symmetric matrix
\begin{equation*}\label{rc1}
  A^{(\mathcal{S})}:=\left[
  \begin{array}{rrrr}
    0 & G & R & B \\
    -G & 0 & B & -R \\
    -R & -B & 0 & G \\
    -B &R & -G & 0 \\
  \end{array}
\right],
\end{equation*}
where 
$$R=\left[
  \begin{array}{cccc}
 0 & 0 & 1 & 0\\
 0 & 0 & 0 & 1\\
 1 & 0 & 0 & 0\\
 0 & 1 & 0 & 0\\
  \end{array}
\right]\otimes Z,~G=\left[
  \begin{array}{cccc}
 0& 1 & 0 & 0\\
  0 & 0 & 1 &0\\
   0 &0 & 0&1 \\
   1& 0&0&0 \\
  \end{array}
\right] \otimes Z,~B=\left[
  \begin{array}{cccc}
 0& 0 & 0 & 1\\
  1 & 0 & 0 &0\\
   0 &1 & 0&0 \\
   0& 0&1&0 \\
  \end{array}
\right]\otimes Z,$$
represent the  red, green, and blue components, respectively, and $Z$ is  an $n$-by-$n$ matrix with all entries being ones. 
Recall that  Sir W. Hamilton (1805-1865) invented quaternions in 1843 when he attended to extend a complex number to a higher spatial dimension space\cite{Hamilton1969}. 
 In the contemporary era, quaternions have been well known  and widely applied in color image processing \cite{ jns19nla, jnw19, S.C.Pei1997, S.J.Sangwine1996}.     See Section \ref{s:cip} for  more details of the application to color image processing.

 \begin{figure}[!h]
  \begin{center}
  \includegraphics[height=3.7cm, width=3.6cm]{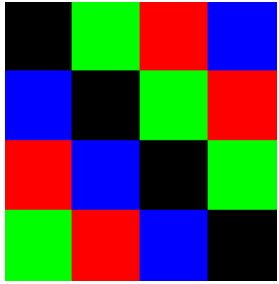}
 \end{center}
 \caption{Color image of size $200\times 200$.}\label{fig:toy}
\end{figure}

The multi-symplectic transformations can preserve the JRS-symmetric structure,
which often associates to the practical relationships between considered elements, 
  during the resolvent process. One example is 
  the orthogonally \emph{$JRS$}-symplectic matrix  $W\in\mathbb{R}^{4n\times 4n}$, proposed in \cite{jwzc18},  having the following form:
\begin{equation*}\label{e:Wstr}
  W:=\left[
  \begin{array}{cccc}
   W^{(0)} & W^{(2)} & W^{(1)} & W^{(3)} \\
    -W^{(2)} & W^{(0)} & W^{(3)} & -W^{(1)} \\
    -W^{(1)} & -W^{(3)} & W^{(0)} & W^{(2)} \\
    -W^{(3)}& W^{(1)} & -W^{(2)} & W^{(0)} \\
  \end{array}
\right], W^{(i)}\in\mathbb{R}^{n\times n}, i=0,1,2,3.
\end{equation*}
Such  a orthogonal and \emph{$JRS$}-symplectic matrix  corresponds to a  unitary quaternion matrix \cite{jwzc18}, and  preserves the $JRS$-symmetric structure in a  similarity transformation.

The following notation will be used frequently.  The superscript $*^{(\mathcal{S})}$ over a matrix $M$, say $M^{(\mathcal{S})}$,   indicates that  $M$ has the algebraic  structure  \eqref{e:Mstr}.  The JRS-symmetric matrix $M^{(\mathcal{S})}$ will be saved by its first  block row  in the practical implementation, i.e., $M^{(\mathcal{S})}:=[M^{(0)},M^{(2)},M^{(1)},M^{(3)}]$.
The block diagonal matrix with diagonal elements, $A_i$'s, is denoted by ${\tt diag}(A_1,A_2,\cdots,A_k)$.

\section{Multi-Symplectic Lanczos Method}\label{s:MSL}
In this section, we present the multi-symplectic Lanczos method for the computation of the $k~(k< n)$ largest and smallest singular triplets of a $4n$-by-$4n$ JRS-symmetric matrices. The core work is to calculate a structured bidiagonalization of the form 
\begin{equation}\label{e:BS}
B_{k}^{(\mathcal{S})}={\tt diag}( B_{k}, B_{k}, B_{k}, B_{k}),
     ~B_{k}=\left[
          \begin{array}{ccccc}
            \alpha_{1} & \beta_{1} & 0 & \cdots & 0 \\
            0 & \alpha_{2} & \beta_{2}  & \ddots & 0 \\
            \vdots & \ddots & \ddots & \ddots& \vdots \\
            \vdots & \ddots & \ddots & \ddots & \beta_{k-1}  \\
            0 & \cdots & \cdots & 0 & \alpha_{k} \\
          \end{array}
        \right]\in\mathbb{R}^{k\times k}, 
        \end{equation}
    rather than a bidiagonal matrix of size $4k\times 4k$ by the classic Lanczos method.  
   
\subsection{Multi-Symplectic Lanczos Bidiagonalization}\label{s:MSLBi}
The multi-symplectic Lanczos bidigonalization method aims to compute  the partial bidiagonalization of a JRS-symmetric matrix $M\in \mathbb{R}^{4m\times 4n}$.

  At first, we propose the  bidigonalization of a JRS-symmetric matrix $M\in \mathbb{R}^{4m\times 4n}$.  
\begin{theorem}[Multi-Symplectic Bidiagonalization Decomposition]\label{t:MSB}
If matrix $M\in\mathbb{R}^{4m\times 4n}$ is \emph{$JRS$}-symmetric. Then there exist two orthogonally multi-symplectic matrices $P\in\mathbb{R}^{4n\times4n}$ and $Q\in\mathbb{R}^{4m\times4m}$ such that
 \begin{equation}\label{rpld}
    M=Q\left[
      \begin{array}{cccc}
        B & 0 & 0 & 0 \\
        0 & B & 0 & 0 \\
        0 & 0 & B & 0 \\
        0 & 0 & 0 & B \\
      \end{array}
    \right]P^T,
     \end{equation}
where
$B\in\mathbb{R}^{m\times n}$ is a real bidiagonal matrix.
\end{theorem}
\begin{proof} We prove the theorem by induction on the dimension. 
Clearly, the theorem holds  for $m=n=1$. Assume that the result is true for  any $4(m-1)\times 4(n-1)$ \emph{$JRS$}-symmetric matrix.

For any $JRS$-symmetric matrix $M\in \mathbb{R}^{4m\times 4n}$, its four distant blocks are partitioned as 
 \begin{equation*}M^{(0)}=\left[\begin{array}{c|ccc}
                                 \omega^{(0)}_{11} &  \omega^{(0)}_{12}  &\Omega^{(0)}_{13} \\ \hline
                                  \omega^{(0)}_{21} &  \omega^{(0)}_{22} &\Omega^{(0)}_{23}\\
                                     \Omega^{(0)}_{31} & \Omega^{(0)}_{32} & \Omega^{(0)}_{33}\\
                             \end{array}
                             \right],\ \ \ \ \ \
M^{(1)}=\left[\begin{array}{c|ccc}
                                \omega^{(1)}_{11} &  \omega^{(1)}_{12}&  \Omega^{(1)}_{13}   \\ \hline
                                  \omega^{(1)}_{21} &  \omega^{(1)}_{22} &  \Omega^{(1)}_{23} \\
                                    \Omega^{(1)}_{31} &  \Omega^{(1)}_{32} &  \Omega^{(1)}_{33} \\
                              \end{array}
                             \right],
 \end{equation*}
 \begin{equation*}M^{(2)}=\left[\begin{array}{c|ccc}
                                  \omega^{(2)}_{11} &  \omega^{(2)}_{12}&  \Omega^{(2)}_{13} \\ \hline
                                  \omega^{(2)}_{21} &  \omega^{(2)}_{22} &  \Omega^{(2)}_{23} \\
                                    \Omega^{(2)}_{31} &  \Omega^{(2)}_{32} &  \Omega^{(2)}_{33} \\
                                \end{array}
                             \right],\ \ \ \ \ \
M^{(3)}=\left[\begin{array}{c|ccc}
                                   \omega^{(3)}_{11} &  \omega^{(3)}_{12}&  \Omega^{(3)}_{13}   \\ \hline
                                  \omega^{(3)}_{21} &  \omega^{(3)}_{22} &  \Omega^{(3)}_{23} \\
                                    \Omega^{(3)}_{31} &  \Omega^{(3)}_{32} &  \Omega^{(3)}_{33} \\
                               \end{array}
                             \right],
 \end{equation*}
 in which $\omega^{(t)}_{ij}\in\mathbb{R}$,   $\Omega^{(t)}_{33} \in \mathbb{R}^{(m-2)\times (n-2)}$, and $\Omega^{(t)}_{i3}, (\Omega^{(t)}_{3j})^T\in\mathbb{R}^{m-2}$. There exist two  muti-symplectic matrices, denoted by $P_1\in\mathbb{R}^{4n\times 4n}$ and $Q_1\in\mathbb{R}^{4m\times 4m}$ such that 
\begin{align*}
\tilde{M}:=Q_1MP_1^T :=\left[
              \begin{array}{cccc}
               \tilde{{M}}^{(0)}& \tilde{{M}}^{(2)}& \tilde{{M}}^{(1)}& \tilde{{M}}^{(3)}\\
               -\tilde{{M}}^{(2)}&  \tilde{{M}}^{(0)}& \tilde{{M}}^{(3)} & -\tilde{{M}}^{(1)} \\
                -\tilde{{M}}^{(1)}& -\tilde{{M}}^{(3)} & \tilde{{M}}^{(0)}& \tilde{{M}}^{(2)} \\
               - \tilde{{M}}^{(3)}&\tilde{{M}}^{(1)} & -\tilde{{M}}^{(2)} &  \tilde{{M}}^{(0)}\\
              \end{array}
            \right]
\end{align*}
 where
\begin{equation*}\tilde{{M}}^{(0)}=\left[\begin{array}{c|cc}
                                 \gamma_{11} &  \gamma_{12}& 0  \\ \hline
                                0&  \tilde{{\omega}}^{(0)}_{22} & \tilde{ {\Omega}}^{(0)}_{23} \\
                                   0& \tilde{{ \Omega}}^{(0)}_{32} &  \tilde{{\Omega}}^{(0)}_{33}
                             \end{array}
                             \right],\ \ \
\tilde{{M}}^{(1)}=\left[\begin{array}{c|cc}
                                 0 &  0& 0   \\ \hline
                                  0 &  \tilde{{\omega}}^{(1)}_{22} &  \tilde{{\Omega}}^{(1)}_{23} \\
                                   0 &  \tilde{{\Omega}}^{(1)}_{32} &   \tilde{{\Omega}}^{(1)}_{33}    
                               \end{array}
                             \right],
 \end{equation*}
\begin{equation*}\tilde{{M}}^{(2)}=\left[\begin{array}{c|cc}
                                 0 &  0& 0  \\ \hline
                                0&  \tilde{{\omega}}^{(2)}_{22} & \tilde{ {\Omega}}^{(2)}_{23} \\
                                   0& \tilde{{ \Omega}}^{(2)}_{32} &  \tilde{{\Omega}}^{(2)}_{33}
                             \end{array}
                             \right],\ \ \
\tilde{{M}}^{(3)}=\left[\begin{array}{c|cc}
                                 0 &  0& 0   \\ \hline
                                  0 &  \tilde{{\omega}}^{(3)}_{22} &  \tilde{{\Omega}}^{(3)}_{23} \\
                                   0 &  \tilde{{\Omega}}^{(3)}_{32} &   \tilde{{\Omega}}^{(3)}_{33}    
                               \end{array}
                             \right].
 \end{equation*}
Here, $ \gamma_{11}, \gamma_{12}\ge 0$, and  $P_1$ and $Q_1$ are products of a series of generalized Givens transformations defined as in \cite{jwl13},
\begin{equation*}\label{e}
G(i):=\left[\begin{smallmatrix}
\emph{I}_{i-1},&           0,& 0,& 0,&           0,& 0,& 0,&           0,& 0,& 0,&            0,& 0\\
0,&  \cos\alpha_0,& 0,& 0,&  \cos\alpha_2,& 0,& 0,&   \cos\alpha_1,& 0,& 0,&   \cos\alpha_3,& 0\\
0,&           0,& \emph{I}_{m-i},& 0,&           0,& 0,& 0,&           0,& 0,& 0,&            0,& 0\\
0,&           0,& 0,& \emph{I}_{i-1},&           0,& 0,& 0,&           0,& 0,& 0,&            0,& 0\\
0,& -\cos\alpha_2,& 0,& 0,&  \cos\alpha_0,& 0,& 0,&  \cos\alpha_3,& 0,& 0,&   -\cos\alpha_1,& 0\\
0,&           0,& 0,& 0,&           0,& \emph{I}_{m-i},& 0,&           0,& 0,& 0,&            0,& 0\\
0,&           0,& 0,& 0,&           0,& 0,& \emph{I}_{i-1},&           0,& 0,& 0,&            0,& 0\\
0,&  -\cos\alpha_1,& 0,& 0,& -\cos\alpha_3,& 0,& 0,&  \cos\alpha_0,& 0,& 0,&   \cos\alpha_2,& 0\\
0,&           0,& 0,& 0,&           0,& 0,& 0,&           0,& \emph{I}_{m-i},& 0,&            0,& 0\\
0,&           0,& 0,& 0,&           0,& 0,& 0,&           0,& 0,& \emph{I}_{i-1},&            0,& 0\\
0,& -\cos\alpha_3,& 0,& 0,&   \cos\alpha_1,& 0,& 0,& -\cos\alpha_2,& 0,& 0,&   \cos\alpha_0,& 0\\
0,&           0,& 0,& 0,&           0,& 0,& 0,&           0,& 0,& 0,&            0,& \emph{I}_{m-i}\\
\end{smallmatrix}\right],
\end{equation*}
where $\alpha_0, \alpha_1,\alpha_2, \alpha_3\in[-\pi/2,\pi/2)$ and $\cos^2\alpha_0+\cos^2\alpha_1+\cos^2\alpha_2+\cos^2\alpha_3=1$.

Denote $M'$ as the submatrix of $\tilde{{M}}$ by deleting the $1,m+1,2m+1,3m+1$ rows and the $1,n+1,2n+1,3n+1$ columns. Clearly,  $M'\in \mathbb{R}^{4(m-1)\times4(n-1)}$ is $JRS$-symmetric. By the introduction assumption, 
there exist two orthogonal $JRS$-symplectic matrices,
$$
  P'=\left[\begin{array}{cccc}
  P'^{(0)} & P'^{(2)} & P'^{(1)}& P'^{(3)} \\
    -P'^{(2)} & P'^{(0)} & P'^{(3)} & -P'^{(1)} \\
    -P'^{(1)}& -P'^{(3)} & P'^{(0)} & P'^{(2)} \\
    -P'^{(3)} & P'^{(1)} & -P'^{(2)} & P'^{(0)} \\
 \end{array}\right]$$
and
$$Q'=\left
  [\begin{array}{cccc}
     Q'^{(0)} & Q'^{(2)} & Q'^{(1)}& Q'^{(3)} \\
    -Q'^{(2)} & Q'^{(0)} & Q'^{(3)} & -Q'^{(1)} \\
    -Q'^{(1)}& -Q'^{(3)} & Q'^{(0)} & Q'^{(2)} \\
    -Q'^{(3)} & Q'^{(1)} & -Q'^{(2)} & Q'^{(0)} \\
  \end{array}\right]
$$
 such that\\
 \begin{equation}\label{subM}
(Q')^{T}M' P'= \left[
      \begin{array}{cccc}
        B' & 0 & 0 & 0 \\
        0 & B' & 0 & 0 \\
        0 & 0 & B'& 0 \\
        0 & 0 & 0 & B' \\
      \end{array}
    \right],
 \end{equation}
in which  $B'\in \mathbb{R}^{(m-1)\times (n-1)}$ is a bidiagonal matrix.
Define 
$$P_2^{(t)}=\left[\begin{array}{cccc}
               1 & 0 \\
                0 & P'^{(t)}\\
              \end{array}\right]\in\mathbb{R}^{n\times n}~\text{and}~Q_2^{(t)}=\left[\begin{array}{cc}
               1 & 0 \\
                0 & Q'^{(t)}\\
              \end{array}\right]\in\mathbb{R}^{m\times m}, ~t=0,1,2,3.$$
 We structure two orthogonal multi-symplectic matrices
\begin{equation*}
P_2:=\left[
             \begin{array}{cccc}
             P_2^{(0)} & P_2^{(2)}&P_2^{(1)} & P_2^{(3)} \\
                -P_2^{(2)}& P_2^{(0)} &P_2^{(3)} & -P_2^{(1)} \\
                -P_2^{(1)}& -P_2^{(3)}&P_2^{(0)} & P_2^{(2)} \\
                -P_2^{(3)}& P_2^{(1)} &-P_2^{(2)} & P_2^{(0)} \\
              \end{array}
            \right]
            \end{equation*}
       and
       \begin{equation*}
       Q_2:=\left[
              \begin{array}{cccc}
               Q_2^{(0)} & Q_2^{(2)}&Q_2^{(1)} & Q_2^{(3)} \\
                -Q_2^{(2)}& Q_2^{(0)} &Q_2^{(3)} & -Q_2^{(1)} \\
                -Q_2^{(1)}& -Q_2^{(3)}&Q_2^{(0)} & Q_2^{(2)} \\
                -Q_2^{(3)}& Q_2^{(1)} &-Q_2^{(2)} & Q_2^{(0)} \\
              \end{array}
            \right].
\end{equation*}
Then by defining $P=P_1P_2$ and $Q=Q_1Q_2$, we have
\begin{equation}\label{l}
   Q^TMP=(Q_1Q_2)^{T}M P_1P_2= \left[
      \begin{array}{cccc}
        B & 0 & 0 & 0 \\
        0 & B& 0 & 0 \\
        0 & 0 & B & 0 \\
        0 & 0 & 0 & B \\
      \end{array}
    \right],
\end{equation}
where 
$$B=\left[\begin{array}{c|c}\gamma_{11}&[\gamma_{12}~0]\\ \hline
                                                0 & B'
\end{array}\right]\in\mathbb{R}^{m\times n}$$
 is a bidiagonal matrix.   $P$ and $Q$ are surely orthogonally multi-symplectic, since they are products of two orthogonally multi-symplectic matrices. 
\end{proof}

Next, we propose the multi-symplectic Lanczos bidigonalization (MLB) method  to compute  the partial bidiagonalization of a JRS-symmetric matrix $M\in \mathbb{R}^{4m\times 4n}$.  
 \begin{theorem}[Partial  Bidiagonalization Decomposition]\label{t:PMSB} Suppose that $M^{(\mathcal{S})}\in \mathbb{R}^{4m\times 4n}$ is a JRS-symmetric matrix   of the form \eqref{e:Mstr} and $1\le k\le \min(m,n)$. There exist two multi-symmplectic matrices $P_k^{(\mathcal{S})}\in\mathbb{R}^{4n\times 4k}$ and $Q_k^{(\mathcal{S})}\in\mathbb{R}^{4m\times 4k}$ with orthogonal columns, such that 
 \begin{equation}\label{rpld.pa}
  M^{(\mathcal{S})}P_{k}^{(\mathcal{S})}=Q_{k}^{(\mathcal{S})}B_{k}^{(\mathcal{S})},~ (M^{(\mathcal{S})})^{T}Q_{k}^{(\mathcal{S})}=P_{k}^{(\mathcal{S})}B_{k}^{(\mathcal{S})}+r_{k}^{(\mathcal{S})}(e_{k}^{(\mathcal{S})})^{T},
\end{equation}
where $B_{k}^{(\mathcal{S})}$ is of the form \eqref{e:BS} with $\alpha_{j}>0$ and $\beta_{j}\geq0 $,    $r_{k}^{(\mathcal{S})}=\texttt{diag}(r_{k},r_{k},r_{k},r_{k})$ $\in \mathbb{R}^{4n\times 4}$ is a ``residual vector'' satisfying $(P_{k}^{(\mathcal{S})})^T r_{k}^{(\mathcal{S})}=0$, and $e_{k}^{(\mathcal{S})}=\texttt{diag}(e_{k},e_{k},e_{k},e_{k})$ $\in \mathbb{R}^{4m\times 4}$ with $e_{k}$ denoting  the $k$th column of the identity matrix. 
\end{theorem}
\begin{proof}
With  the structure-preserving transformations in hand, we can complete the proof in the similar way to that of Theorem \ref{t:MSB}. 
\end{proof}

The pseudo-code of the MLB method is proposed in Algorithm 1, in which $\|\cdot\|$ denotes the $Frobenius$ $norm$.
To save  storage, we use the notation $M^{(\mathcal{S})}$, $P_k^{(\mathcal{S})}$ and $Q_k^{(\mathcal{S})}$,  but only update and save their first block rows. 

Notice that we need the  reorthogonalization with preserving multi-symplectic structure in lines 5 and 10  since the orthogonality of computed columns is often lost during computation. If without breakdown, we compute matrices $P_k^{(\mathcal{S})}$, $Q_k^{(\mathcal{S})}$ and $B_k^{(\mathcal{S})}$, satisfying Theorem \ref{t:PMSB} .
The output matrix $B_k$ is a bidiagonal matrix as in \eqref{e:BS},   $\alpha_{j}>0,\beta_{j}\ge 0, 1 \leq j \leq k$. 

Due to the limitation of computing speed and memory, the value of $k$ should not be too large  in practical calculation. 
However,  too small $k$ can not always guarantee the enough accuracy and the convergence of the multi-symplectic Lanczos bidiagonalization algorithm. A useful technique is restarting, which is applied to make the approximate singular triplets converge at the preset accuracy. At present, there are many available restarted strategies.  The implicitly restarting technology proposed by Sorensen \cite{Sorensen} for eigenvalues is one of the most successful restarting technologies. We will propose the details of applying this technique in Section \ref{s:augritz}. 

\begin{algorithm} \label{code:SPLB}
{{\bf Algorithm 1.} The Multi-Symplectic Lanczos (MSL) Bidiagonalization }
\begin{enumerate}
              \item[Input:\ ] $M^{(\mathcal{S})}\in \mathbb{R}^{m \times 4n}$: \ a $JRS$-symmetric matrix,
              \item[] $p_{1}^{(\mathcal{S})} \in \mathbb{R}^{n\times 4}$:\ ``initial vector'' satisfies $\|p_{1}^{(\mathcal{S})}\|=1$,
              \item[] $k$:\ number of bidiagonalization steps.
              \item[Output:] $P_{k}^{(\mathcal{S})} \in\mathbb{R}^{n\times 4k}$:\ matrix with orthonormal columns,
              \item[] $Q_{k}^{(\mathcal{S})}\in\mathbb{R}^{m\times 4k}$:\ matrix with orthonormal columns,
              \item[] $B_{k}\in \mathbb{R}^{k\times k}$:\ upper bidiagonal matrix with entries $\alpha_{j},\beta_{j}$,
              \item[] $r_{k}^{(\mathcal{S})} \in \mathbb{R}^{n \times 4}$:\ residual vector.
\item $P_{1}^{(\mathcal{S})}:=p_{1}^{(\mathcal{S})};q_{1}^{(\mathcal{S})}:=Mp_{1}^{(\mathcal{S})}$;
\item $\alpha_{1}:=\parallel q_{1}^{(\mathcal{S})} \parallel;  q_{1}^{(\mathcal{S})}:=q_{1}^{(\mathcal{S})}/\alpha_{1}; Q_{1}^{(\mathcal{S})}:=q_{1}^{(\mathcal{S})};$
 \item    for {$j=1:k$}
\item \quad $r_{j}^{(\mathcal{S})}:=(M^{(\mathcal{S})})^{T}q_{j}^{(\mathcal{S})}-\alpha_{j}p_{j}^{(\mathcal{S})};$
\item \quad  Reorthogonalization: $r_{j}^{(\mathcal{S})}:=r_{j}^{(\mathcal{S})}-P_{j}^{(\mathcal{S})}((P_{j}^{(\mathcal{S})})^{T}r_{j}^{(\mathcal{S})});$
 \item \quad if  {$j<k,$}
\item \quad\quad $\beta_{j}:=\parallel r_{j}^{(\mathcal{S})} \parallel; \ p_{j+1}^{(\mathcal{S})}:=r_{j}^{(\mathcal{S})}/\beta_{j}$; 
\item  \quad\quad $P_{j+1}^{(\mathcal{S})}:$=$[P_{j}^{(\mathcal{S})}(:,1:j),p_{j+1}^{(\mathcal{S})}(:,1),P_{j}^{(\mathcal{S})}(:,j+1:2j),$ $p_{j+1}^{(\mathcal{S})}(: ,2)$, $P_{j}^{(\mathcal{S})}(:,2j+1:3j),$  $p_{j+1}^{(\mathcal{S})}(:,3),P_{j}^{(\mathcal{S})}(:,3j+1:4j),p_{j+1}^{(\mathcal{S})}(:,4)];$
\item \quad\quad $q_{j+1}^{(\mathcal{S})}:=M^{(\mathcal{S})}p_{j+1}^{(\mathcal{S})}-\beta_{j}q_{j}^{(\mathcal{S})}$;
\item \quad\quad Reorthogonalization:\ $q_{j+1}^{(\mathcal{S})}:=q_{j+1}^{(\mathcal{S})}-Q_{j}^{(\mathcal{S})}(Q_{j}^{(\mathcal{S})})^{T}q_{j+1}^{(\mathcal{S})}$;
\item \quad\quad  $\alpha_{j+1}:=\|q_{j+1}^{(\mathcal{S})}\|;\ q_{j+1}^{(\mathcal{S})}:=q_{j+1}^{(\mathcal{S})}/\alpha_{j+1};$\ 
\item \quad\quad $Q_{j+1}^{(\mathcal{S})}:=[Q_{j}^{(\mathcal{S})}(:,1:j),$ $q_{j+1}^{(\mathcal{S})}(:,1),$ $Q_{j}^{(\mathcal{S})}(:,j+1:2j),$ $q_{j+1}^{(\mathcal{S})}(:,2),$  $Q_{j}^{(\mathcal{S})}(:,2j+1:3j),$ $q_{j+1}^{(\mathcal{S})}(:,3),$ $Q_{j}^{(\mathcal{S})}(:,3j+1:4j),q_{j+1}^{(\mathcal{S})}(:,4)]; $
\item \quad end
\item end
\end{enumerate}
\end{algorithm}

\subsection{Computation of $k$  Largest Singular Triplets}\label{s:augritz}
Let the partial Lanczos bidiagonalization \eqref{rpld.pa} be available, and assume that we are interested in determining the $t$ largest singular triplets of $M\in\mathbb{R}^{4m\times 4n}$, where $t\leq k<\min(m,n)$.  
Let $\{\sigma_{j},u_{j},v_{j}\}$, $ j=1,\dots, k $,  with 
\begin{equation*}\label{svd.Bk.a}
  \sigma_{1}\geq \sigma_{2}\geq \cdots \geq \sigma_{k} \geq 0,
\end{equation*}
be the singular triplets of $B_{k}$,
i.e., 
\begin{equation}\label{svd.Bk}
  B_{k}v_{j}=u_{j}\sigma_{j}, \ B_{k}^{T}u_{j}=v_{j}\sigma_{j}.
\end{equation}
To obtain the singular triplets of $M^{(\mathcal{S})}$, we need firstly construct three diagonal or block diagonal  matrices as follows:
\begin{equation}\label{appro.rA.a}
  \sigma_{j}^{(\mathcal{S})}=\texttt{diag}(\sigma_{j},\sigma_{j},\sigma_{j},\sigma_{j}),
  ~ u_{j}^{(\mathcal{S})}=\texttt{diag}(u_{j},u_{j},u_{j},u_{j}),~
   v_{j}^{(\mathcal{S})}=\texttt{diag}(v_{j},v_{j},v_{j},v_{j}).
\end{equation}
The right and left singular vectors of $M^{(\mathcal{S})}$ are generated by computing  
\begin{equation}\label{appro.rA.b}
  \tilde{u}_{j}^{(\mathcal{S})}=Q_{k}^{(\mathcal{S})}u_{j}^{(\mathcal{S})}, ~
 \tilde{v}_{j}^{(\mathcal{S})}=P_{k}^{(\mathcal{S})}v_{j}^{(\mathcal{S})}.
\end{equation}
Then we finally obtain the Lanczos bidiagonalization of $M^{(\mathcal{S})}$, 
\begin{equation}\label{svd.rA}
  M^{(\mathcal{S})}\tilde{v}_{j}^{(\mathcal{S})}= \tilde{u}_{j}^{(\mathcal{S})}\sigma_{j}^{(\mathcal{S})}, \ (M^{(\mathcal{S})})^{T}\tilde{u}_{j}^{(\mathcal{S})}=\tilde{v}_{j}^{(\mathcal{S})}\sigma_{j}^{(\mathcal{S})}+r_{k}^{(\mathcal{S})}(e_{k}^{(\mathcal{S})})^{T}u_{j}^{(\mathcal{S})},\  1\leq j\leq k.
\end{equation}
So that  $\{{\tilde{\sigma}_{j}^{(\mathcal{S})},\tilde{u}_{j}^{(\mathcal{S})},\tilde{v}_{j}^{(\mathcal{S})}}\}$ can be accepted as an approximate singular triplet of $M^{(\mathcal{S})}$ if the Frobenius norm of $r_{k}^{(\mathcal{S})}(e_{k}^{(\mathcal{S})})^{T}u_{j}^{(\mathcal{S})}$ is sufficiently small.

From the partial Lanczos bidigonalization decomposition of  $M^{(\mathcal{S})}$, we can exactly  deduce the partial Lanczos tridiagonalization of $(M^{(\mathcal{S})})^{T}M^{(\mathcal{S})}$.  Indeed,  multiplying the first equation of \eqref{rpld.pa} by $(M^{(\mathcal{S})})^{T}$ from the left-hand side and applying the second equation of  \eqref{rpld.pa} yields
\begin{equation}\label{ATAPLT}
  (M^{(\mathcal{S})})^{T}M^{(\mathcal{S})}P_{k}^{(\mathcal{S})}
  =P_{k}^{(\mathcal{S})}(B_{k}^{(\mathcal{S})})^{T}B_{k}^{(\mathcal{S})}+\alpha_{k}r_{k}^{(\mathcal{S})}(e_{k}^{(\mathcal{S})})^{T}.
\end{equation}
The columns of $P_{k}^{(\mathcal{S})}$ are separated into four independent groups which satisfy the three-term recurrence relationship. This can be seen from  
\begin{equation*}\label{rTk}
  T_{k}^{(\mathcal{S})}:=(B_{k}^{(\mathcal{S})})^{T}B_{k}^{(\mathcal{S})}\\
  =\texttt{diag}(T_{k},T_{k},T_{k},T_{k}),
\end{equation*}
where
\begin{equation*}\label{Tk}
  T_{k}:=B_{k}^{T}B_{k}=\left[
          \begin{array}{ccccc}
            \alpha_{1}^{2} & \beta_{1}\alpha_{1} & 0 & \cdots & 0 \\
            \beta_{1}\alpha_{1} & \alpha_{2}^{2}+\beta_{1}^{2} & \beta_{2}\alpha_{3} & \cdots & 0 \\
            0 & \ddots & \ddots & \ddots & 0 \\
            0 & 0 & \ddots & \alpha_{k-1}^{2}+\beta_{k-2}^{2} & \beta_{k-2}\alpha_{k} \\
            0 & 0 & \cdots & \beta_{k-2}\alpha_{k} & \alpha_{k}^{2}+\beta_{k-1}^{2}  \\
          \end{array}
        \right]\in \mathbb{R}^{k\times k}.
\end{equation*}
Define $p_{j}^{(\mathcal{S})}:=P_{k}^{(\mathcal{S})}e_{j}^{(\mathcal{S})}$, $j=1,\ldots,k$. 
Then \{$p_{1}^{(\mathcal{S})}$,\ $p_{2}^{(\mathcal{S})}$,\ \ldots,\ $p_{k}^{(\mathcal{S})}$\} acts as an orthonormal basis of the block Krylov subspace
\begin{equation*}\label{KATA}
   \mathcal{K}((M^{(\mathcal{S})})^{T}M^{(\mathcal{S})},p_{1}^{(\mathcal{S})} )=\texttt{span}\{p_{1}^{(\mathcal{S})} ,(M^{(\mathcal{S})})^{T}M^{(\mathcal{S})}p_{1}^{(\mathcal{S})},\cdots,((M^{(\mathcal{S})})^{T}M^{(\mathcal{S})})^{k-1}p_{1}^{(\mathcal{S})} \}.
\end{equation*}
Similarly, multiplying the second equation of  \eqref{rpld.pa} by $M^{(\mathcal{S})}$ from the left-hand side and applying the first equation of  \eqref{rpld.pa} yield
\begin{equation*}\label{AATPLT}
    M^{(\mathcal{S})}(M^{(\mathcal{S})})^{T}Q_{k}^{(\mathcal{S})}=Q_{k}^{(\mathcal{S})}B_{k}^{(\mathcal{S})}(B_{k}^{(\mathcal{S})})^{T}+M^{(\mathcal{S})}r_{k}^{(\mathcal{S})}(e_{k}^{(\mathcal{S})})^{T}.
\end{equation*}
Define $q_{j}^{(\mathcal{S})}:=Q_{k}^{(\mathcal{S})}e_{j}^{(\mathcal{S})}$, $j=1,\ldots,k$.
Then \{$q_{1}^{(\mathcal{S})}$,\ $q_{2}^{(\mathcal{S})}$,\ \ldots,\ $q_{k}^{(\mathcal{S})}$\} is an orthonormal basis of the block Krylov subspace
\begin{equation*}\label{KAAT}
  \mathcal{K}(M^{(\mathcal{S})}(M^{(\mathcal{S})})^{T},q_{1}^{(\mathcal{S})})=\texttt{span}\{q_{1}^{(\mathcal{S})} ,M^{(\mathcal{S})}(M^{(\mathcal{S})})^{T}q_{1}^{(\mathcal{S})},\cdots,(M^{(\mathcal{S})}(M^{(\mathcal{S})})^{T})^{k-1}q_{1}^{(\mathcal{S})} \}.
\end{equation*}
Defined as in \eqref{appro.rA.b},   $\tilde{v}_{j}^{(\mathcal{S})}$   is also called a ``Ritz vector'' of $(M^{(\mathcal{S})})^{T}M^{(\mathcal{S})}$ associated with the ``Ritz value'' $(\tilde{\sigma}_{j}^{(\mathcal{S})})^{2}$.
In fact,  multiplying \eqref{ATAPLT} by $v_{j}^{(\mathcal{S})}$ from the right-hand side  yields
\begin{equation*}\label{Ritz.1}
  (M^{(\mathcal{S})})^{T}M^{(\mathcal{S})}\tilde{v}_{j}^{(\mathcal{S})}-\tilde{v}_{j}^{(\mathcal{S})}(\sigma_{j}^{(\mathcal{S})})^{2}=\alpha_{k}r_{k}^{(\mathcal{S})}(e_{k}^{(\mathcal{S})})^{T}v_{j}^{(\mathcal{S})}, \ 1\leq j\leq k.
\end{equation*}

Based on above analysis,  we can present the restarted Lanczos bidiagonalization method  with structure preservation.
Suppose we have computed the Ritz vectors $\tilde{v}_{j}^{(\mathcal{S})}$, $1\leq j\leq t$,  associated with the first $t$ largest Ritz values. These Ritz vectors are gathered into a JRS-symmetric matrix
\begin{equation*}
  \tilde{V}_{t}^{(\mathcal{S})}:=\left[
                           \begin{array}{rrrr}
                             \tilde{V}_{t}^{(0)} & \tilde{V}_{t} ^{(2)} & \tilde{V}_{t} ^{(1)} & \tilde{V}_{t} ^{(3)} \\
                             -\tilde{V}_{t}^{(2)} & \tilde{V}_{t} ^{(0)} &\tilde{V}_{t} ^{(3)} & -\tilde{V}_{t}^{(1)} \\
                             -\tilde{V}_{t} ^{(1)} & -\tilde{V}_{t}^{(3)} & \tilde{V}_{t} ^{(0)} & \tilde{V}_{t} ^{(2)} \\
                             -\tilde{V}_{t} ^{(3)} & \tilde{V}_{t} ^{(1)} & -\tilde{V}_{t}^{(2)} & \tilde{V}_{t} ^{(0)} \\
                           \end{array}
                         \right]\in \mathbb{R}^{4n\times 4t}.
\end{equation*}
Suppose that $\beta_{k}>0$ in line $7$ of Algorithm 1, then  the $(k+1)$th block column vector of  $P_{k+1}^{(\mathcal{S})}$ is
\begin{equation}\label{e:pk+1}
  p_{k+1}^{(\mathcal{S})}=r_{k}^{(\mathcal{S})}/\beta_{k}.
\end{equation} 
Matrix $\tilde{V}_{t}^{(\mathcal{S})}$ is enlarged to 
\begin{equation*}
  \tilde{V}_{t+1}^{(\mathcal{S})}:=\left[
                                       \begin{array}{rrrr}
                              \tilde{V}_{t+1} ^{(0)} & \tilde{V}_{t+1} ^{(2)} & \tilde{V}_{t+1} ^{(1)} & \tilde{V}_{t+1} ^{(3)} \\
                            - \tilde{V}_{t+1} ^{(2)} & \tilde{V}_{t+1} ^{(0)} & \tilde{V}_{t+1} ^{(3)} & - \tilde{V}_{t+1} ^{(1)} \\
                             - \tilde{V}_{t+1} ^{(1)} & - \tilde{V}_{t+1} ^{(3)} & \tilde{V}_{t+1} ^{(0)}& \tilde{V}_{t+1} ^{(2)} \\
                             - \tilde{V}_{t+1} ^{(3)} & \tilde{V}_{t+1} ^{(1)} & - \tilde{V}_{t+1} ^{(2)} & \tilde{V}_{t+1} ^{(0)}\\
                                       \end{array}
                                     \right]\in \mathbb{R}^{4n\times 4(t+1)},
\end{equation*}
where
$$ \tilde{V}_{t+1} ^{(0)}=[\tilde{V}_{t} ^{(0)},p_{k+1}^{(\mathcal{S})}(1:n,1)],~
 \tilde{V}_{t+1} ^{(2)}=[\tilde{V}_{t}^{(2)},p_{k+1}^{(\mathcal{S})}(1:n,2)],$$
$$ \tilde{V}_{t+1}^{(1)}=[\tilde{V}_{t} ^{(1)},p_{k+1}^{(\mathcal{S})}(1:n,3)],~
\tilde{V}_{t+1} ^{(3)}=[\tilde{V}_{t}^{(3)},p_{k+1}^{(\mathcal{S})}(1:n,4)].$$
According to \eqref{e:pk+1}, the last column of each block of $ \tilde{V}_{t+1}^{(\mathcal{S})}$ is  a ``vector''  parallel to the residual error $r_{k}^{(\mathcal{S})}$.
Restarting the Lanczos process with $p_{k+1}^{(\mathcal{S})}$ as the new initial vector and utilizing  $(\ref{appro.rA.a})$-$(\ref{appro.rA.b})$ and $(\ref{svd.rA})$,  we obtain $M^{(\mathcal{S})}{ \tilde{V}_{t+1}}^{(\mathcal{S})}:=$ with blocks as 
$$(M^{(\mathcal{S})}\tilde{V}_{t+1}^{(\mathcal{S})})^{(0)}=[ (\tilde{u}_{1}^{(\mathcal{S})}\tilde{\sigma}_{1}^{(\mathcal{S})})(1:m,1),\cdots, (\tilde{u}_{t}^{(\mathcal{S})}\tilde{\sigma}_{t}^{(\mathcal{S})})(1:m,1),(M^{(\mathcal{S})}p_{k+1}^{(\mathcal{S})})(1:m,1)],$$
$$(M^{(\mathcal{S})}\tilde{V}_{t+1}^{(\mathcal{S})})^{(2)}=[ (\tilde{u}_{1}^{(\mathcal{S})}\tilde{\sigma}_{1}^{(\mathcal{S})})(1:m,2),\cdots, (\tilde{u}_{t}^{(\mathcal{S})}\tilde{\sigma}_{t}^{(\mathcal{S})})(1:m,2),(M^{(\mathcal{S})}p_{k+1}^{(\mathcal{S})})(1:m,2)],$$
$$(M^{(\mathcal{S})}\tilde{V}_{t+1}^{(\mathcal{S})})^{(1)}=[ (\tilde{u}_{1}^{(\mathcal{S})}\tilde{\sigma}_{1}^{(\mathcal{S})})(1:m,3),\cdots, (\tilde{u}_{t}^{(\mathcal{S})}\tilde{\sigma}_{t}^{(\mathcal{S})})(1:m,3),(M^{(\mathcal{S})}p_{k+1}^{(\mathcal{S})})(1:m,3)],$$
$$(M^{(\mathcal{S})}\tilde{V}_{t+1}^{(\mathcal{S})})^{(3)}=[ (\tilde{u}_{1}^{(\mathcal{S})}\tilde{\sigma}_{1}^{(\mathcal{S})})(1:m,4),\cdots, (\tilde{u}_{t}^{(\mathcal{S})}\tilde{\sigma}_{t}^{(\mathcal{S})})(1:m,4),(M^{(\mathcal{S})}p_{k+1}^{(\mathcal{S})})(1:m,4)].$$

Now we orthogonalize $M^{(\mathcal{S})}p_{k+1}^{(\mathcal{S})}$ to  each $\tilde{u}_{j}^{(\mathcal{S})}$ and get 
\begin{equation}\label{Ritz.5}
  \tilde{r}_{t}^{(\mathcal{S})}=M^{(\mathcal{S})}p_{k+1}^{(\mathcal{S})}-\sum_{j=1}^{t}(\tilde{u}_{j}^{(\mathcal{S})})^{T}M^{(\mathcal{S})}p_{k+1}^{(\mathcal{S})} \tilde{u}_{j}^{(\mathcal{S})}.
\end{equation}
Since
 \begin{equation}\label{Ritz.6}
  (\tilde{u}_{j}^{(\mathcal{S})})^{T}M^{(\mathcal{S})}p_{k+1}^{(\mathcal{S})}=\beta_{k}(e_{k}^{(\mathcal{S})})^{T}u_{j}^{(\mathcal{S})},
\end{equation}
is  diagonal,
we can define $\tilde{\rho}_{j}^{(\mathcal{S})}:=(\tilde{u}_{j}^{(\mathcal{S})})^{T}M^{(\mathcal{S})}p_{k+1}^{(\mathcal{S})}:=\texttt{diag}(\tilde{\rho}_{j},\tilde{\rho}_{j},\tilde{\rho}_{j},\tilde{\rho}_{j})$.
The remainder vector $\tilde{r}_{t}^{(\mathcal{S})}$ is always assumed to be nonzero.  Otherwise, the iteration is terminated.
We normalize $\tilde{r}_{t}^{(\mathcal{S})}$ and add it to $\tilde{U}_{t}^{(\mathcal{S})}$ as the last column, which generates an enlarged matrix 
\begin{equation*}\label{Ritz.7}
  \tilde{U}_{t+1}^{(\mathcal{S})}:=\left[
                                     \begin{array}{rrrr}
                                       \tilde{U}_{t+1}^{(0)} & \tilde{U}_{t+1} ^{(2)} & \tilde{U}_{t+1} ^{(1)} & \tilde{U}_{t+1} ^{(3)} \\

                                       - \tilde{U}_{t+1}^{(2)} &  \tilde{U}_{t+1}^{(0)} &  \tilde{U}_{t+1}^{(3)} & - \tilde{U}_{t+1}^{(1)} \\

                                       - \tilde{U}_{t+1}^{(1)} & - \tilde{U}_{t+1}^{(3)} &  \tilde{U}_{t+1}^{(0)} &  \tilde{U}_{t+1}^{(2)} \\

                                       - \tilde{U}_{t+1}^{(3)} &  \tilde{U}_{t+1}^{(1)} & - \tilde{U}_{t+1}^{(2)} &  \tilde{U}_{t+1}^{(0)} \\
                                     \end{array}
                                   \right]\in \mathbb{R}^{4m\times 4(t+1)},
\end{equation*}
where
$$ \tilde{U}_{t+1}^{(0)}=[\tilde{u}_{1}^{(\mathcal{S})} (1:m;1),\cdots,\tilde{u}_{t}^{(\mathcal{S})}(1:m;1),(\tilde{r}_{t}^{(\mathcal{S})}/\|\tilde{r}_{t}^{(\mathcal{S})}\|)(1:m;1)],$$
$$ \tilde{U}_{t+1}^{(2)}=[\tilde{u}_{1}^{(\mathcal{S})}(1:m;2),\cdots,\tilde{u}_{t}^{(\mathcal{S})}(1:m;2),(\tilde{r}_{t}^{(\mathcal{S})}/\|\tilde{r}_{t}^{(\mathcal{S})}\|)(1:m;2)],$$
$$ \tilde{U}_{t+1}^{(1)}=[\tilde{u}_{1}^{(\mathcal{S})}(1:m;3),\cdots,\tilde{u}_{t}^{(\mathcal{S})}(1:m;3),(\tilde{r}_{t}^{(\mathcal{S})}/\|\tilde{r}_{t}^{(\mathcal{S})}\|)(1:m;3)],$$
$$ \tilde{U}_{t+1}^{(3)}=[\tilde{u}_{1}^{(\mathcal{S})}(1:m;4),\cdots,\tilde{u}_{t}^{(\mathcal{S})}(1:m;4),(\tilde{r}_{t}^{(\mathcal{S})}/\|\tilde{r}_{t}^{(\mathcal{S})}\|)(1:m;4)].$$
Define a $4(t+1)$-by-$4(t+1)$ matrix
\begin{equation}\label{Ritz.8}
  \tilde{B}_{t+1}^{(\mathcal{S})}:=\texttt{diag}( \tilde{B}_{t+1} , \tilde{B}_{t+1} , \tilde{B}_{t+1} , \tilde{B}_{t+1} )
\end{equation}
with
\begin{equation*}\label{Ritz.9}
  \tilde{B}_{t+1}:=\left[
                     \begin{array}{cccc}
                       \tilde{\sigma}_{1} &   & 0 & \tilde{\rho}_{1}  \\
                         & \ddots &   &\vdots\\
                         &   & \tilde{\sigma}_{t} & \tilde{\rho}_{t} \\
                     0 &   &  & \tilde{\alpha}_{t+1} \\
                     \end{array}
                   \right]\in \mathbb{R}^{(t+1)\times (t+1)}.
\end{equation*}
Then 
\begin{equation}\label{Ritz.10}
  M^{(\mathcal{S})}{ \tilde{V}_{t+1}}^{(\mathcal{S})}={\tilde{U}_{t+1}}^{(\mathcal{S})} \tilde{B}_{t+1}^{(\mathcal{S})}.
\end{equation}
Next we try to express $(M^{(\mathcal{S})})^{T}\tilde{U}_{t+1}^{(\mathcal{S})}$ in terms of ${ \tilde{V}_{t+1}}^{(\mathcal{S})}$ and $(\tilde{B}_{t+1}^{(\mathcal{S})})^{T}$.
According to $(\ref{svd.rA})$ and $(\ref{Ritz.6})$,
\begin{equation}\label{Ritz.11}
  (M^{(\mathcal{S})})^{T}\tilde{u}_{j}^{(\mathcal{S})}=\tilde{v}_{j}^{(\mathcal{S})}\tilde{\sigma}_{j}^{(\mathcal{S})}+r_{k}^{(\mathcal{S})}(e_{k}^{(\mathcal{S})})^{T}u_{j}^{(B_{k})}=\tilde{v}_{j}^{(\mathcal{S})}\tilde{\sigma}_{j}^{(\mathcal{S})}+p_{k+1}^{(\mathcal{S})}\tilde{\rho}_{j}^{(\mathcal{S})}, \  1\leq j\leq t.
\end{equation}
The last block column, $(M^{(\mathcal{S})})^{T}(\tilde{r}_{t}^{(\mathcal{S})}/\|\tilde{r}_{t}^{(\mathcal{S})}\|)$,  is orthogonal to the Ritz vectors $\tilde{v}_{j}^{(\mathcal{S})}$, i.e.,
 \begin{equation*}\label{Ritz.12}
  (\tilde{v}_{j}^{(\mathcal{S})})^{T}(M^{(\mathcal{S})})^{T}(\tilde{r}_{t}^{(\mathcal{S})}/\|\tilde{r}_{t}^{(\mathcal{S})}\|)=
  (\tilde{u}_{j}^{(\mathcal{S})})^{T}\tilde{\sigma}_{j}^{(\mathcal{S})}(\tilde{r}_{t}^{(\mathcal{S})}/\|\tilde{r}_{t}^{(\mathcal{S})}\|)=0, \ 1\leq j\leq t.
 \end{equation*}
So it can be expressed as 
\begin{equation}\label{Ritz.13}
  (M^{(\mathcal{S})})^{T}(\tilde{r}_{t}^{(\mathcal{S})}/\|\tilde{r}_{t}^{(\mathcal{S})}\|)=\tilde{\delta} p_{k+1}^{(\mathcal{S})}+\tilde{f}_{t+1}^{(\mathcal{S})},
\end{equation}
where
$\tilde{f}_{t+1}^{(\mathcal{S})}\in \mathbb{R}^{4n\times 4}$ is orthogonal to the vectors $\tilde{v}_{j}^{(\mathcal{S})}$, $1\leq j\leq t$, as well as to $p_{k+1}^{(\mathcal{S})}$ .
Using formula (\ref{Ritz.5}), we have
\begin{equation*}\label{Ritz.14}
  (p_{k+1}^{(\mathcal{S})})^{T} (M^{(\mathcal{S})})^{T}(\tilde{r}_{t}^{(\mathcal{S})}/\|\tilde{r}_{t}^{(\mathcal{S})}\|)=\|\tilde{r}_{t}^{(\mathcal{S})}\|.
\end{equation*}
Hence, $\tilde{\delta}=\|\tilde{r}_{t}^{(\mathcal{S})}\|$.
Combining (\ref{Ritz.11}) with (\ref{Ritz.13}) induces  the following expression
\begin{equation}\label{Ritz.15}
  (M^{(\mathcal{S})})^{T}\tilde{U}_{t+1}^{(\mathcal{S})}= \tilde{V}_{t+1}^{(\mathcal{S})}(\tilde{B}_{t+1}^{(\mathcal{S})})^{T}+\tilde{f}_{t+1}^{(\mathcal{S})}(e_{t+1}^{T})^{(\mathcal{S})}.
\end{equation}
We remark that $\tilde{f}_{t+1}^{(\mathcal{S})}$ can be computed by $(\ref{Ritz.13})$.

 If necessary, we restart the Lanczos bidiagonalization with the enlarged column as the new initial vector.
Suppose that  $\tilde{f}_{t+1}^{(\mathcal{S})}\neq 0$.  Let $\tilde{\beta}_{t+1}:=\|\tilde{f}_{t+1}^{(\mathcal{S})}\|$
and $\tilde{p}_{t+2}^{(\mathcal{S})}:=\tilde{f}_{t+1}^{(\mathcal{S})}/\tilde{\beta}_{t+1}$.
Then we can enlarge $ \tilde{V}_{t+1}^{(\mathcal{S})}$ into
\begin{equation*}\label{Ritz.16}
  \tilde{V}_{t+2}^{(\mathcal{S})}:=\left[
                                     \begin{array}{rrrr}
                                       \tilde{V}_{t+2}^{(0)} & \tilde{V}_{t+2} ^{(2)} & \tilde{V}_{t+2} ^{(1)} & \tilde{V}_{t+2} ^{(3)} \\
                                       -\tilde{V}_{t+2} ^{(2)} & \tilde{V}_{t+2}^{(0)} & \tilde{V}_{t+2}^{(3)} & -\tilde{V}_{t+2}^{(1)} \\
                                       -\tilde{V}_{t+2}^{(1)} & -\tilde{V}_{t+2}^{(3)} & \tilde{V}_{t+2}^{(0)} & \tilde{V}_{t+2}^{(2)} \\
                                       -\tilde{V}_{t+2}^{(3)} & \tilde{V}_{t+2}^{(1)} & -\tilde{V}_{t+2}^{(2)} & \tilde{V}_{t+2}^{(0)} \\
                                     \end{array}
                                   \right]\in \mathbb{R}^{4n\times4(t+2)},
\end{equation*}
where
\begin{align*}
\tilde{V}_{t+2}^{(0)}&=[ \tilde{V}_{t+1}^{(\mathcal{S})}(1:n,1:t+1),\tilde{p}_{t+2}^{(\mathcal{S})}(1:n,1)],\\
\tilde{V}_{t+2}^{(2)}&=[ \tilde{V}_{t+1}^{(\mathcal{S})}(1:n,t+2:2(t+1)),\tilde{p}_{t+2}^{(\mathcal{S})}(1:n,2)],\\
\tilde{V}_{t+2}^{(1)}&=[ \tilde{V}_{t+1}^{(\mathcal{S})}(1:n,2(t+1)+1:3(t+1)),\tilde{p}_{t+2}^{(\mathcal{S})}(1:n,3)],\\ 
\tilde{V}_{t+2}^{(3)}&=[ \tilde{V}_{t+1}^{(\mathcal{S})}(1:n,3(t+1)+1:4(t+1)),\tilde{p}_{t+2}^{(\mathcal{S})}(1:n,4)].
\end{align*}
Let
\begin{equation}\label{Ritz.17}
  \tilde{\alpha}_{t+2}\tilde{q}_{t+2}^{(\mathcal{S})}:=(I_{4m}-\tilde{U}_{t+1}^{(\mathcal{S})}(\tilde{U}_{t+1}^{(\mathcal{S})})^{T})M\tilde{p}_{t+2}^{(\mathcal{S})},
\end{equation}
where $\tilde{\alpha}_{t+2}>0$ is a scaling factor, such that  $\tilde{q}_{t+2}^{(\mathcal{S})}$ is of unit length; the Frobenius norm is $1$.
Then equation $(\ref{Ritz.15})$ yields
\begin{equation}\label{Ritz.19}
\begin{split} \tilde{\alpha}_{t+2}\tilde{q}_{t+2}^{(\mathcal{S})}
&=M^{(\mathcal{S})}\tilde{p}_{t+2}^{(\mathcal{S})}-\tilde{U}_{t+1}^{(\mathcal{S})}((M^{(\mathcal{S})})^{T}\tilde{U}_{t+1}^{(\mathcal{S})})^{T}\tilde{p}_{t+2}^{(\mathcal{S})}\\
&=M^{(\mathcal{S})}\tilde{p}_{t+2}^{(\mathcal{S})}-\tilde{U}_{t+1}^{(\mathcal{S})}( \tilde{V}_{t+1}^{(\mathcal{S})}(\tilde{B}_{t+1}^{(\mathcal{S})})^{T}+\tilde{\beta}_{t+1}\tilde{p}_{t+2}^{(\mathcal{S})}(e_{t+1}^{(\mathcal{S})})^{T})^{T}\tilde{p}_{t+2}^{(\mathcal{S})}\\
&=M^{(\mathcal{S})}\tilde{p}_{t+2}^{(\mathcal{S})}-\tilde{U}_{t+1}^{(\mathcal{S})}(\tilde{B}_{t+1}^{(\mathcal{S})}( \tilde{V}_{t+1}^{(\mathcal{S})})^{T}+\tilde{\beta}_{t+1}e_{t+1}^{(\mathcal{S})}(\tilde{p}_{t+2}^{(\mathcal{S})})^{T})\tilde{p}_{t+2}^{(\mathcal{S})}\\
&=M^{(\mathcal{S})}\tilde{p}_{t+2}^{(\mathcal{S})}-\tilde{\beta}_{t+1}\tilde{U}_{t+1}^{(\mathcal{S})}e_{t+1}^{(\mathcal{S})}\\
&=M^{(\mathcal{S})}\tilde{p}_{t+2}^{(\mathcal{S})}-\tilde{\beta}_{t+1}\tilde{U}_{t+1}^{(\mathcal{S})}.\\
\end{split}
\end{equation}
In similar way, we construct
\begin{equation*}\label{Ritz.20}
  \tilde{U}_{t+2}^{(\mathcal{S})}:=\left[
                                     \begin{array}{rrrr}
                                       \tilde{U}_{t+2} ^{(0)} & \tilde{U}_{t+2}^{(2)} & \tilde{U}_{t+2}^{(1)} & \tilde{U}_{t+2}^{(3)} \\
                                       -\tilde{U}_{t+2}^{(2)} & \tilde{U}_{t+2}^{(0)} & \tilde{U}_{t+2}^{(3)} & -\tilde{U}_{t+2}^{(1)} \\
                                       -\tilde{U}_{t+2}^{(1)} & -\tilde{U}_{t+2}^{(3)} & \tilde{U}_{t+2}^{(0)} & \tilde{U}_{t+2}^{(2)} \\
                                       -\tilde{U}_{t+2}^{(3)} & \tilde{U}_{t+2}^{(1)} & -\tilde{U}_{t+2}^{(2)} & \tilde{U}_{t+2}^{(0)} \\
                                     \end{array}
                                   \right]\in \mathbb{R}^{4m\times4(t+2)}
\end{equation*}
with 
\begin{align*}
\tilde{U}_{t+2}^{(0)}&=[\tilde{U}_{t+1}^{(\mathcal{S})}(1:m,1:t+1),\tilde{q}_{t+2}^{(\mathcal{S})}(1:m,1)],\\
\tilde{U}_{t+2}^{(2)}&=[\tilde{U}_{t+1}^{(\mathcal{S})}(1:m,t+2:2(t+1)),\tilde{q}_{t+2}^{(\mathcal{S})}(1:m,2)],\\
\tilde{U}_{t+2}^{(1)}&=[\tilde{U}_{t+1}^{(\mathcal{S})}(1:m,2(t+1)+1:3(t+1)),\tilde{q}_{t+2}^{(\mathcal{S})}(1:m,3)],\\
\tilde{U}_{t+2}^{(3)}&=[\tilde{U}_{t+1}^{(\mathcal{S})}(1:m,3(t+1)+1:4(t+1)),\tilde{q}_{t+2}^{(\mathcal{S})}(1:m,4)].
\end{align*}
It now follows from (\ref{Ritz.10}) and (\ref{Ritz.19}) that
\begin{equation}\label{Ritz.23}
  M^{(\mathcal{S})}{\tilde{V}_{t+2}}^{(\mathcal{S})}={\tilde{U}_{t+2}}^{(\mathcal{S})} \tilde{B}_{t+2}^{(\mathcal{S})},
\end{equation}
where
\begin{equation*}\label{Ritz.21}
  \tilde{B}_{t+2}^{(\mathcal{S})}:
  =\texttt{diag}( \tilde{B}_{t+2}, \tilde{B}_{t+2}, \tilde{B}_{t+2}, \tilde{B}_{t+2}),~
   \tilde{B}_{t+2}:=\left[
                     \begin{array}{ccccc}
                       \tilde{\sigma}_{1} &   & 0 & \tilde{\rho}_{1} & 0 \\
                         & \ddots &   & \vdots & \vdots \\
                         &   & \tilde{\sigma}_{t} & \tilde{\rho}_{t} & 0 \\
                         &   &   & \tilde{\alpha}_{t+1} & \tilde{\beta}_{t+1} \\
                       0 &   &   &   & \tilde{\alpha}_{t+2} \\
                     \end{array}
                   \right].
\end{equation*}
Let
\begin{equation}\label{Ritz.24}
\begin{split}
  \tilde{\beta}_{t+2}\tilde{p}_{t+3}^{(\mathcal{S})}:=(I_{4n}-\tilde{V}_{t+2}^{(\mathcal{S})}( \tilde{V}_{t+2}^{(\mathcal{S})})^{T})M^{T}\tilde{q}_{t+2}^{(\mathcal{S})},
  \end{split}
\end{equation}
where $\tilde{\beta}_{t+2}>0$ is a scaling factor such that  $\tilde{p}_{t+3}^{(\mathcal{S})}$ is of unit length. 
 Substituting $(\ref{Ritz.23})$ into $(\ref{Ritz.24})$ yields
\begin{equation}\label{Ritz.25}
  \begin{split}
   \tilde{\beta}_{t+2}\tilde{p}_{t+3}^{(\mathcal{S})}&=(M^{(\mathcal{S})})^{T}\tilde{q}_{t+2}^{(\mathcal{S})}-\tilde{V}_{t+2}^{(\mathcal{S})}(M\tilde{V}_{t+2}^{(\mathcal{S})})^{T}\tilde{q}_{t+2}^{(\mathcal{S})}\\
                                                &=(M^{(\mathcal{S})})^{T}\tilde{q}_{t+2}^{(\mathcal{S})}-\tilde{V}_{t+2}^{(\mathcal{S})}(\tilde{B}_{t+2}^{(\mathcal{S})})^{T}(\tilde{U}_{t+2}^{(\mathcal{S})})^{T}\tilde{q}_{t+2}^{(\mathcal{S})}\\
                                                &=(M^{(\mathcal{S})})^{T}\tilde{q}_{t+2}^{(\mathcal{S})}-\tilde{V}_{t+2}^{(\mathcal{S})}(\tilde{B}_{t+2}^{(\mathcal{S})})^{T}e_{t+2}^{(\mathcal{S})}\\
                                                &=(M^{(\mathcal{S})})^{T}\tilde{q}_{t+2}^{(\mathcal{S})}-\tilde{\alpha}_{t+2}\tilde{p}_{t+2}^{(\mathcal{S})}.\\
  \end{split}
\end{equation}
Thus, 
\begin{equation*}\label{Ritz.26}
   (M^{(\mathcal{S})})^{T}\tilde{U}_{t+2}^{(\mathcal{S})}=\tilde{V}_{t+2}^{(\mathcal{S})}(\tilde{B}_{t+2}^{(\mathcal{S})})^{T}+\tilde{\beta}_{t+2}\tilde{p}_{t+3}^{(\mathcal{S})}(e_{t+2}^{(\mathcal{S})})^{T}.
\end{equation*}
After  $k-t$ similar steps, we obtain the decompositions
\begin{equation*}\label{Ritz.27}
    M^{(\mathcal{S})}{\tilde{V}_{k}}^{(\mathcal{S})}={\tilde{U}_{k}}^{(\mathcal{S})} \tilde{B}_{k}^{(\mathcal{S})},\\
 (M^{(\mathcal{S})})^{T}\tilde{U}_{k}^{(\mathcal{S})}=\tilde{V}_{k}^{(\mathcal{S})}(\tilde{B}_{k}^{(\mathcal{S})})^{T}+\tilde{\beta}_{k}\tilde{p}_{k+1}^{(\mathcal{S})}(e_{k}^{(\mathcal{S})})^{T},
\end{equation*}
where $\tilde{V}_{k}^{(\mathcal{S})}$ and $\tilde{U}_{k}^{(\mathcal{S})}$ have orthonormal columns, and $\tilde{B}_{k}^{(\mathcal{S})}=\texttt{diag}(\tilde{B}_{k},\tilde{B}_{k},\tilde{B}_{k},\tilde{B}_{k})$
with 
$$\tilde{B}_{k}=\left[
                     \begin{array}{cccccc}
                       \tilde{\sigma}_{1} &   & 0 & \tilde{\rho}_{1} &  & 0\\
                         & \ddots &   & \vdots &  &  \\
                         &   & \tilde{\sigma}_{t} & \tilde{\rho}_{t} &   &  \\
                         &   &   & \tilde{\alpha}_{t+1} & \tilde{\beta}_{t+1} &  \\
                         &   &   &   & \ddots  & \\
                         &   &   &   &    \ddots &\tilde{\beta}_{k-1} \\
                       0 &   &   &   &   & \tilde{\alpha}_{k}\\
                     \end{array}
                   \right].$$
At the end  of each cycle,  we calculate the singular value decomposition of $\tilde{B}_{k}$ to get  the approximations of $t$ largest singular triplets of $M^{(\mathcal{S})}$.

\subsection{Computation of $k$ Smallest Singular Triplets}\label{augHritz}
The shifting by harmonic Ritz values can be implemented via augmentation by harmonic Ritz vectors. In this section,  we turn to computing the $t$ smallest singular triplets of a nonsingular JRS-symmetric matrix $M^{(\mathcal{S})}$ by  augmentation with harmonic Ritz vectors.

Suppose that  the partial Lanczos bidiagonalization $\eqref{rpld.pa}$ of $M^{(\mathcal{S})}$ have been  available, and all the diagonal and superdiagonal entries of $B_{k}$, as well as $\beta_{k}$ given by $\beta_{k}:=\parallel r_{k}\parallel$, are nonvanishing. 
\begin{definition}
A value $\hat{\theta}_{j}^{(\mathcal{S})}$ is a harmonic Ritz value of a matrix $M^{(\mathcal{S})}$ with respect to some linear subspace $\mathcal{W}_{k}^{(\mathcal{S})}$ if $(\hat{\theta}_{j}^{(\mathcal{S})})^{-1}$ is a Ritz value of $(M^{(\mathcal{S})})^{-1}$ with respect to $\mathcal{W}_{k}^{(\mathcal{S})}$.   The approximate eigenvectors of $M^{(\mathcal{S})}$ associated with harmonic Ritz values are called  harmonic Ritz vectors.
\end{definition}

 The ``harmonic Ritz values'' of $(M^{(\mathcal{S})})^{T}M^{(\mathcal{S})}$ 
  are the generalized eigenvalues of 
\begin{equation}\label{hr.2}
    \big(\big((B_{k}^{(\mathcal{S})})^{T}B_{k}^{(\mathcal{S})}\big)^2+\alpha_{k}^{2}\beta_{k}^{2}e_{k}^{(\mathcal{S})}(e_{k}^{(\mathcal{S})})^{T}\big)\hat{\omega}_{j}^{(\mathcal{S})}=(B_{k}^{(\mathcal{S})})^{T}(B_{k}^{(\mathcal{S})})\hat{\omega}_{j}^{(\mathcal{S})}\hat{\theta}_{j}^{(\mathcal{S})}.
\end{equation}
Define 
\begin{equation}\label{hr.3}
  \omega_{j}^{(\mathcal{S})}:=B_{k}^{(\mathcal{S})}\hat{\omega}_{j}^{(\mathcal{S})}=\texttt{diag}( B_{k}\hat{\omega}_{j}, B_{k}\hat{\omega}_{j}, B_{k}\hat{\omega}_{j}, B_{k}\hat{\omega}_{j}):=\texttt{diag}( \omega_{j}, \omega_{j}, \omega_{j}, \omega_{j}).
\end{equation}
Then equation \eqref{hr.2} can be expressed as 
\begin{equation}\label{hr.2.1}
    [B_{k}^{(\mathcal{S})}(B_{k}^{(\mathcal{S})})^{T}+\beta_{k}^{2}e_{k}^{(\mathcal{S})}(e_{k}^{(\mathcal{S})})^{T}]\omega_{j}^{(\mathcal{S})}=\omega_{j}^{(\mathcal{S})}\hat{\theta}_{j}^{(\mathcal{S})},
\end{equation}
since $B_{k}^{(\mathcal{S})}$  is invertible. 
The structured equation \eqref{hr.2.1} is equivalent to a reduced  form,
\begin{equation}\label{hr.4}
  (B_{k}B_{k}^{T}+\beta_{k}^{2}e_{k}e_{k}^{T})\omega_{j}=\omega_{j}\hat{\theta}_{j},~\text{i.e.},~ (B_{k,k+1}B_{k,k+1}^{T})\omega_{j}=\omega_{j}\hat{\theta}_{j}
\end{equation}
where  $B_{k,k+1}$ is the leading $k\times (k+1)$ submatrix of $B_{k+1}$. 
Here one can choose the eigenvectors, $\omega_{j}$'s, to be orthonormal. That means that the eigenpairs $\{\hat{\theta}_{j}^{(\mathcal{S})},\hat{\omega}_{j}^{(\mathcal{S})}\}$  can be computed without  forming the matrix $(B_{k}^{(\mathcal{S})})^{T}B_{k}^{(\mathcal{S})}$ explicitly.
We refer to \cite{morgan91,ppv95} and also to \cite{bare05,kbg04} for  the analysis of the eigenvalue problem \eqref{hr.4}.

Suppose that we have computed the singular triplets of matrix $B_{k,k+1}$, denoted by  $(\sigma_{j},u_{j},v_{j})$, where $1\leq j\leq k$,  $0< \sigma_{1}\leq \sigma_{2}\leq \cdots \leq \sigma_{k}$, $u_{j}\in \mathbb{R}^{k+1}$ and $v_{j}\in \mathbb{R}^{k}$.
The $t$ smallest singular triplets of $B_{k,k+1}$ determine the matrices
\begin{equation*}\label{hr.8}
U_{t}:=[u_{1},u_{2},\cdots,u_{t}],~
V_{t}:=[v_{1},v_{2},\cdots,v_{t}],~
\Sigma_{t}:=\texttt{diag}(\sigma_{1},\sigma_{2},\cdots,\sigma_{t}).
\end{equation*}
Clearly,   $(\sigma_{j}^{2},u_{j})$  is the eigenpair of $B_{k,k+1}B_{k,k+1}^{T}$. 
With \eqref{hr.2} and \eqref{hr.3} in mind,  we indeed  obtain $t$ generalized eigenpairs of the equation $(\ref{hr.2})$.
 In other words,   the ``harmonic Ritz values'' of $(M^{(\mathcal{S})})^{T}M^{(\mathcal{S})}$  can be computed by the singular value decomposition of a  bidiagonal matrix $B_{k,k+1}$.  The obtained harmonic Ritz value  and associated harmonic Ritz vector  are exactly
\begin{subequations}\label{hr.10}\begin{align}
&\hat{\theta}_{j}^{(\mathcal{S})}={\tt diag}(\sigma_{j}^{2},\sigma_{j}^{2},\sigma_{j}^{2},\sigma_{j}^{2}),&\\
&  \hat{v}_{j}^{(\mathcal{S})}:=P_{k}^{(\mathcal{S})}\hat{\omega}_{j}^{(\mathcal{S})}=P_{k}^{(\mathcal{S})}　\texttt{diag}(B_{k}^{-1}u_{j},B_{k}^{-1}u_{j},B_{k}^{-1}u_{j},B_{k}^{-1}u_{j}). &
  \end{align}
\end{subequations}
The residual error associated with different harmonic Ritz pair $\{\hat{\theta}_{j}^{(\mathcal{S})},  \hat{v}_{j}^{(\mathcal{S})}\}$ is
$$
\begin{aligned}
&(M^{(\mathcal{S})})^{T}M^{(\mathcal{S})}\hat{v}_{j}^{(\mathcal{S})}-\hat{v}_{j}^{(\mathcal{S})}\hat{\theta}_{j}^{(\mathcal{S})}\\
&\overset{(\ref{ATAPLT})}{=}[P_{k}^{(\mathcal{S})}(B_{k}^{(\mathcal{S})})^{T}B_{k}^{(\mathcal{S})}+\alpha_{k}r_{k}^{(\mathcal{S})}(e_{k}^{(\mathcal{S})})^{T}]\hat{\omega}_{j}^{(\mathcal{S})}-P_{k}^{(\mathcal{S})}\hat{\omega}_{j}^{(\mathcal{S})}\hat{\theta}_{j}^{(\mathcal{S})}\\
&=P_{k}^{(\mathcal{S})}(B_{k}^{(\mathcal{S})})^{-1}[B_{k}^{(\mathcal{S})}(B_{k}^{(\mathcal{S})})^{T}\omega_{j}^{(\mathcal{S})}-\omega_{j}^{(\mathcal{S})}\hat{\theta}_{j}^{(\mathcal{S})}]+r_{k}^{(\mathcal{S})}(e_{k}^{(\mathcal{S})})^{T}\omega_{j}^{(\mathcal{S})}\\
&\overset{(\ref{hr.2.1})}{=}P_{k}^{(\mathcal{S})}(B_{k}^{(\mathcal{S})})^{-1}[-\beta_{k}^{2}e_{k}^{(\mathcal{S})}(e_{k}^{(\mathcal{S})})^{T}]\omega_{j}^{(\mathcal{S})}+r_{k}^{(\mathcal{S})}(e_{k}^{(\mathcal{S})})^{T}\omega_{j}^{(\mathcal{S})}\\
&=[r_{k}^{(\mathcal{S})}-\beta_{k}^{2}P_{k}^{(\mathcal{S})}(B_{k}^{(\mathcal{S})})^{-1}e_{k}^{(\mathcal{S})}](e_{k}^{(\mathcal{S})})^{T}\omega_{j}^{(\mathcal{S})}.
\end{aligned}
$$
It is convenient to define the scaled residual vector as 
\begin{equation}\label{hr.11}
  \hat{r}_{k}^{(\mathcal{S})}:=p_{k+1}^{(\mathcal{S})}-\beta_{k}P_{k}^{(\mathcal{S})}(B_{k}^{(\mathcal{S})})^{-1}e_{k}^{(\mathcal{S})},
\end{equation}
where $p_{k+1}^{(\mathcal{S})}:=r_{k}^{(\mathcal{S})}/\beta_{k}$ according to Algorithm 1.

Now we present the decomposition of  $M$  under  the restarted Lanczos bidiagonalization with augmentation by harmonic Ritz vectors. 
Define
\begin{equation*}\label{V}
   \hat{V}^{(\mathcal{S})}:=\left[
                                     \begin{array}{rrrr}
                                       \hat{V}^{(0)} &\hat{V}^{(2)} & \hat{V}^{(1)} & \hat{V}^{(3)} \\
                                       -\hat{V}^{(2)} & \hat{V}^{(0)} & \hat{V}^{(3)} & -\hat{V}^{(1)} \\
                                       -\hat{V}^{(1)} & -\hat{V}^{(3)} & \hat{V}^{(0)} & \hat{V}^{(2)} \\
                                       -\hat{V}^{(3)} & \hat{V}^{(1)} & -\hat{V}^{(2)} & \hat{V}^{(0)} \\
                                     \end{array}
                                   \right]\in \mathbb{R}^{4n\times4(t+1)},
\end{equation*}
and
 $C^{(\mathcal{S})}:=\texttt{diag}(C,C,C,C), 
D^{(\mathcal{S})}:=\texttt{diag}(D,D,D,D)\in \mathbb{R}^{4(t+1)\times 4(t+1)}$,
where
$$  \hat{V}^{(0)}=[\hat{v}_{1}^{(\mathcal{S})}(1:n,1),\cdots,\hat{v}_{t}^{(\mathcal{S})}(1:n,1),\hat{r}_{k}^{(\mathcal{S})}(1:n,1)],$$
 $$  \hat{V}^{(2)}=[\hat{v}_{1}^{(\mathcal{S})}(1:n,2),\cdots,\hat{v}_{t}^{(\mathcal{S})}(1:n,2),\hat{r}_{k}^{(\mathcal{S})}(1:n,2)],$$
 $$   \hat{V}^{(1)}=[\hat{v}_{1}^{(\mathcal{S})}(1:n,3),\cdots,\hat{v}_{t}^{(\mathcal{S})}(1:n,3),\hat{r}_{k}^{(\mathcal{S})}(1:n,3)],$$
 $$    \hat{V}^{(3)}=[\hat{v}_{1}^{(\mathcal{S})}(1:n,4),\cdots,\hat{v}_{t}^{(\mathcal{S})}(1:n,4),\hat{r}_{k}^{(\mathcal{S})}(1:n,4)],$$
 $$C=\left[
          \begin{array}{cc}
             B_{k}^{-1}U_{t}\Sigma_{t}& -\beta_{k} B_{k}^{-1}e_{k}\\
            0 & 1 \\
          \end{array}
        \right],\ D=\left[  \begin{array}{ccccc}
                       \sigma_{1} &   & 0 & 0  \\
                         & \ddots &  \vdots & \vdots  \\
                         &   & \sigma_{t} &0\\
                       0  &  \cdots & 0  & 1 \\
                     \end{array}
                   \right].$$
\begin{theorem}
Let
\begin{equation}\label{hr.26}
\hat{B}_{t+1}^{(\mathcal{S})}:={\tt diag}(\hat{D}_{t+1},\hat{D}_{t+1},\hat{D}_{t+1},\hat{D}_{t+1}) (R_{t+1}^{(\mathcal{S})})^{-1},~\hat{D}_{t+1}=\left[
                     \begin{array}{cccc}
                       {\sigma}_{1} &   & 0 & \hat{\gamma}_{1}  \\
                         & \ddots &   &\vdots\\
                         &   & {\sigma}_{t} &\hat{\gamma}_{t}  \\
                     0 &   &  & \hat{\alpha}_{t+1} \\
                     \end{array}
                   \right],
\end{equation}
and 
 \begin{equation*}\label{hr.40}
  \hat{B}_{k}^{(\mathcal{S})}=\left[
                                   \begin{array}{cccc}
                                     \hat{B}_{k} & 0 & 0 & 0 \\
                                     0 & \hat{B}_{k} & 0 & 0 \\
                                     0 & 0 & \hat{B}_{k} & 0 \\
                                     0 & 0 & 0 & \hat{B}_{k} \\
                                   \end{array}
                                 \right],~
\hat{B}_{k}=\left[
               \begin{array}{cccccc}
                 \hat{B}_{t+1} & \hat{\beta}_{t+1} &   &   &   & 0 \\
                   & \hat{\alpha}_{t+2} & \hat{\beta}_{t+2} &   &   &   \\
                   &   & \hat{\alpha}_{t+3} & \hat{\beta}_{t+3} &   &   \\
                   &   &   &   & \ddots &   \\
                   &   &   &   & \ddots & \hat{\beta}_{k-1} \\
                 0 &   &   &   &   & \hat{\alpha}_{k} \\
               \end{array}
             \right].
\end{equation*}
 Then 
\begin{equation}\label{hr.39}
 M\hat{P}_{k}^{(\mathcal{S})}=\hat{Q}_{k}^{(\mathcal{S})}\hat{B}_{k}^{(\mathcal{S})},\  \
 M^{T}\hat{Q}_{k}^{(\mathcal{S})}=\hat{P}_{k}^{(\mathcal{S})}(\hat{B}_{k}^{(\mathcal{S})})^{T}+(\breve{r}_{k}^{(\mathcal{S})})(e_{t+1}^{(\mathcal{S})})^{T},
\end{equation}
 and the residual vector $\breve{r}_{k}^{(\mathcal{S})}$ is orthogonal to the columns of $\hat{P}_{k}^{(\mathcal{S})}$.
\end{theorem}
\begin{proof}

Equations (\ref{hr.3}), (\ref{hr.10}), and (\ref{hr.11}) yield
\begin{equation}\label{hr.12}
  \hat{V}^{(\mathcal{S})}D^{(\mathcal{S})}=P_{k+1}^{(\mathcal{S})}C^{(\mathcal{S})}.
\end{equation}
Let  $C=Q_{t+1}R_{t+1}$ be the QR-factorization of $C$. Then we obtain the  $QR$-factorization
\begin{equation}\label{hr.13}
C^{(\mathcal{S})}=Q_{t+1}^{(\mathcal{S})}R_{t+1}^{(\mathcal{S})},
\end{equation}
where $Q_{t+1}^{(\mathcal{S})}=\texttt{diag}(Q_{t+1},Q_{t+1},Q_{t+1},Q_{t+1})$, $R_{t+1}^{(\mathcal{S})}=\texttt{diag}(R_{t+1},R_{t+1},R_{t+1},R_{t+1})$.

Define  a $4n \times 4(t+1)$ matrix
\begin{equation}\label{hr.15}
  \hat{P}_{t+1}^{(\mathcal{S})}:= P_{k+1}^{(\mathcal{S})}Q_{t+1}^{(\mathcal{S})}.
\end{equation}
According to $Q_{t+1}=CR_{t+1}^{-1}$ from \eqref{hr.13},
$$M\hat{P}_{t+1}^{(\mathcal{S})}:=\left[
                                           \begin{array}{rrrr}
                                             (M\hat{P}_{t+1} )^{(0)}&  (M\hat{P}_{t+1} )^{(2)} &  (M\hat{P}_{t+1} )^{(1)} &  (M\hat{P}_{t+1} )^{(3)}\\
                                              -(M\hat{P}_{t+1} )^{(2)}& (M\hat{P}_{t+1} )^{(0)} &  (M\hat{P}_{t+1} )^{(3)}&  -(M\hat{P}_{t+1} )^{(1)}\\
                                              -(M\hat{P}_{t+1} )^{(1)}&  -(M\hat{P}_{t+1} )^{(3)} &  (M\hat{P}_{t+1} )^{(0)}&  (M\hat{P}_{t+1} )^{(2)}\\
                                              -(M\hat{P}_{t+1} )^{(3)} &  (M\hat{P}_{t+1} )^{(1)} & -(M\hat{P}_{t+1} )^{(2)}&  (M\hat{P}_{t+1} )^{(0)}\\
                                           \end{array}
                                         \right], 
$$
where
\begin{align*}
(M\hat{P}_{t+1})^{(0)}&=[(M{P}_{k}^{(\mathcal{S})})(1:m,1:k),(Mp_{k+1}^{(\mathcal{S})})(1:m,1)]CR_{t+1}^{-1},\\
(M\hat{P}_{t+1})^{(2)}&=[(M{P}_{k}^{(\mathcal{S})})(1:m,k+1:2k),(Mp_{k+1}^{(\mathcal{S})})(1:m,2)]CR_{t+1}^{-1},\\
(M\hat{P}_{t+1})^{(1)}&=[(M{P}_{k}^{(\mathcal{S})})(1:m,2k+1:3k),(Mp_{k+1}^{(\mathcal{S})})(1:m,3)]CR_{t+1}^{-1},\\
(M\hat{P}_{t+1})^{(3)}&=[(M{P}_{k}^{(\mathcal{S})})(1:m,3k+1:4k),(Mp_{k+1}^{(\mathcal{S})})(1:m,4)]CR_{t+1}^{-1}.
\end{align*}
Using equations $\eqref{rpld.pa}$ and $\eqref{hr.13}$ to simplify above equations, we have
\begin{align*}
(M\hat{P}_{t+1})^{(0)}
&=[(Q_{k}^{(\mathcal{S})}U_{t}^{(\mathcal{S})}\Sigma_{t}^{(\mathcal{S})})(1:m,1:k),(Mp_{k+1}^{(\mathcal{S})}-\beta_{k}q_{k}^{(\mathcal{S})})(1:m,1)]R_{t+1}^{-1},\\
(M\hat{P}_{t+1})^{(2)}
&=[(Q_{k}^{(\mathcal{S})}U_{t}^{(\mathcal{S})}\Sigma_{t}^{(\mathcal{S})})(1:m,k+1:2k),(Mp_{k+1}^{(\mathcal{S})}-\beta_{k}q_{k}^{(\mathcal{S})})(1:m,2)]R_{t+1}^{-1},\\
(M\hat{P}_{t+1})^{(1)}
&=[(Q_{k}^{(\mathcal{S})}U_{t}^{(\mathcal{S})}\Sigma_{t}^{(\mathcal{S})})(1:m,2k+1:3k),(Mp_{k+1}^{(\mathcal{S})}-\beta_{k}q_{k}^{(\mathcal{S})})(1:m,3)]R_{t+1}^{-1}\\
(M\hat{P}_{t+1})^{(3)}
&=[(Q_{k}^{(\mathcal{S})}U_{t}^{(\mathcal{S})}\Sigma_{t}^{(\mathcal{S})})(1:m,3k+1:4k),(Mp_{k+1}^{(\mathcal{S})}-\beta_{k}q_{k}^{(\mathcal{S})})(1:m,4)]R_{t+1}^{-1}.
\end{align*}
Here $U_{t}^{(\mathcal{S})}:=\texttt{diag}(U_{t},U_{t},U_{t},U_{t})$,  $\Sigma_{t}^{(\mathcal{S})}:=\texttt{diag}(\Sigma_{t},\Sigma_{t},\Sigma_{t},\Sigma_{t})$.

The columns of $\hat{Q}_{t}^{(\mathcal{S})}:={Q}_{k}^{(\mathcal{S})}U_{t}^{(\mathcal{S})}$ are orthonormal.
 We define 
\begin{equation}\label{hr.21}
  \hat{c}_{t}=[\hat{\gamma}_{1},\hat{\gamma}_{2},\cdots,\hat{\gamma}_{t}]^{T},\ \hat{c}_{t}^{(\mathcal{S})}=\texttt{diag}(\hat{c}_{t},\hat{c}_{t},\hat{c}_{t},\hat{c}_{t}),
\end{equation}
\begin{equation*}
 \hat{c}_{t}^{(\mathcal{S})} :=(\hat{Q}_{t}^{(\mathcal{S})})^{T}(-\beta_{k}q_{k}^{(\mathcal{S})} +Mp_{k+1}^{(\mathcal{S})} ).
\end{equation*}
The vector 
\begin{equation}\label{hr.22}
 \hat{\alpha}_{t+1}\hat{q}_{t+1}^{(\mathcal{S})} :=-\beta_{k}q_{k}^{(\mathcal{S})}+Mp_{k+1}^{(\mathcal{S})} -\hat{Q}_{t}^{(\mathcal{S})} \hat{c}_{t}^{(\mathcal{S})}
\end{equation}
is orthogonal to the columns of $\hat{Q}_{t}^{(\mathcal{S})} $, and the scaling factor $\hat{\alpha}_{t+1}>0$ is chosen so that $\hat{q}_{t+1}^{(\mathcal{S})} $ is unitary. 
It follows that
\begin{equation}\label{hr.24}
  M\hat{P}_{t+1}^{(\mathcal{S})}=\hat{Q}_{t+1}^{(\mathcal{S})}\hat{B}_{t+1}^{(\mathcal{S})},
\end{equation}
where $\hat{Q}_{t+1}^{(\mathcal{S})}$ is JRS-symmetric and generated by
\begin{align*}
(\hat{Q}_{t+1}^{(\mathcal{S})})^{(0)}&=[\hat{Q}_{t}^{(\mathcal{S})}(1:m,1:t), \ \hat{q}_{t+1}^{(\mathcal{S})}(1:m,1)],\\
(\hat{Q}_{t+1}^{(\mathcal{S})})^{(2)}&=[\hat{Q}_{t}^{(\mathcal{S})}(1:m,t+1:2t), \ \hat{q}_{t+1}^{(\mathcal{S})}(1:m,2)],\\
(\hat{Q}_{t+1}^{(\mathcal{S})})^{(1)}&=[\hat{Q}_{t}^{(\mathcal{S})}(1:m,2t+1:3t), \ \hat{q}_{t+1}^{(\mathcal{S})}(1:m,3)],\\
(\hat{Q}_{t+1}^{(\mathcal{S})})^{(3)}&=[\hat{Q}_{t}^{(\mathcal{S})}(1:m,3t+1:4t), \ \hat{q}_{t+1}^{(\mathcal{S})}(1:m,4)].
\end{align*}
We see that $\hat{B}_{t+1}^{(\mathcal{S})}$ is the product of two upper triangular matrices, one of which has nonzero entries only on the diagonal and in the last column.  In particular, its  diagonal blocks  are upper triangular. 

Similar to  to $(\ref{Ritz.15})$, we can derive an analogue  decomposition
\begin{equation}\label{hr.28}
  M^{T}{\hat{Q}}_{t}^{(\mathcal{S})}=M^{T}{Q}_{t}^{(\mathcal{S})}U_{t}^{(\mathcal{S})}=P_{k+1}^{(\mathcal{S})}(B_{k,k+1}^{(\mathcal{S})})^{T}U_{t}^{(\mathcal{S})}=P_{k+1}^{(\mathcal{S})}(V_{t})^{(\mathcal{S})}\Sigma_{t}^{(\mathcal{S})}.
\end{equation}
Multiplying
\begin{equation*}\label{hr.29}
  B_{k,k+1}V_{t}=[B_{k},\beta_{k}e_{k}]V_{t}=U_{t}\Sigma_{t}
\end{equation*}
 by $B_{k}^{-1}$ from the left-hand side, we get
\begin{equation}\label{hr.30}
  [I_{k},\beta_{k}B_{k}^{-1}e_{k}]V_{t}=B_{k}^{-1}U_{t}\Sigma_{t}.
\end{equation}
So that 
\begin{equation}\label{hr.31}
  V_{t}=\left[
           \begin{array}{cc}
             B_{k}^{-1}U_{t}\Sigma_{t} & -\beta_{k}B_{k}^{-1}e_{k} \\
             0 & 1 \\
           \end{array}
         \right]\left[
                        \begin{array}{c}
                          I_{t} \\
                          e_{k+1}^{T}V_{t} \\
                        \end{array}
                      \right].
\end{equation}
Substituting  \eqref{hr.31} into \eqref{hr.28} and  applying \eqref{hr.13}  yield
\begin{equation}\label{hr.34}
  M^{T}{\hat{Q}}_{t}^{(\mathcal{S})}=\hat{P}_{t+1}^{(\mathcal{S})}(\hat{B}_{t,t+1}^{(\mathcal{S})})^{T},
\end{equation}
where $\hat{B}_{t,t+1}$ is the leading $t\times(t+1)$ submatrix of the upper triangular matrix $\hat{B}_{t+1}$ in $(\ref{hr.24})$.

Considering the last block column of each block of $M^{T}{\hat{Q}}_{t+1}^{(\mathcal{S})}$. Multiplying $M^{T}\hat{q}_{t+1}^{(\mathcal{S})}$ by $(\hat{P}_{t+1}^{(\mathcal{S})})^{T}$ from the left-hand side, \eqref{hr.24} shows that
\begin{equation}\label{hr.35}
  (\hat{P}_{t+1}^{(\mathcal{S})})^{T}M^{T}\hat{q}_{t+1}^{(\mathcal{S})}=(\hat{B}_{t+1}^{(\mathcal{S})})^{T}(\hat{Q}_{t+1}^{(\mathcal{S})})^{T}\hat{q}_{t+1}^{(\mathcal{S})}=(\hat{B}_{t+1}^{(\mathcal{S})})^{T}e_{t+1}^{(\mathcal{S})}=\hat{\alpha}_{t+1}e_{t+1}^{(\mathcal{S})},
\end{equation}
where $\hat{\alpha}_{t+1}$ denotes the last diagonal entry of each block of $\hat{B}_{t+1}^{(\mathcal{S})}$. Thus,
\begin{equation}\label{hr.36}
  M^{T}\hat{q}_{t+1}^{(\mathcal{S})}=\hat{\alpha}_{t+1}\hat{p}_{t+1}^{(\mathcal{S})}+\breve{r}_{t+1}^{(\mathcal{S})},
\end{equation}
where $\breve{r}_{t+1}^{(\mathcal{S})}=\texttt{diag}(\breve{r}_{t+1},\breve{r}_{t+1},\breve{r}_{t+1},\breve{r}_{t+1})$ and $(\hat{P}_{t+1}^{(\mathcal{S})})^{T}\breve{r}_{t+1}^{(\mathcal{S})}=0$. Combining $(\ref{hr.34})$ and $(\ref{hr.36})$ yields
\begin{equation}\label{hr.38}
   M^{T}\hat{Q}_{t+1}^{(\mathcal{S})}= \hat{P}_{t+1}^{(\mathcal{S})}(\hat{B}_{t+1}^{(\mathcal{S})})^{T}+\breve{r}^{(\mathcal{S})}_{t+1}(e_{t+1}^{(\mathcal{S})})^{T}.
\end{equation}

Applying the decompositions $(\ref{hr.24})$ and $(\ref{hr.38})$, we can proceed  to compute the decompositions
\eqref{hr.39}.
\end{proof}

\section{Applications to Color Image Processing}\label{s:cip}
In the color image processing, one of the most important targets is   to compute the optimal low-rank approximation to a color image. Based on the quaternion representation of color image,   such approximation can be reconstructed by a few  of  dominant  singular triplets of a quaternion matrix. 
In this section, we apply the multi-symplectic Lanczos method to compute $k$ largest or smallest singular  triplets of quaternion matrices, which directly leads to its applications to color image processing.

As in \cite{jns19nla,jns19na, jnw19, S.C.Pei1997, S.J.Sangwine1996}, we represent a color image with the spatial resolution of $m\times n$ pixels   by an $m\times n$ pure quaternion matrix,
\begin{equation*}
 \Aq =A^{(0)}+ A^{(1)}\iq + A^{(2)}\jq + A^{(3)}\kq,
\end{equation*}
where $A^{(0)}=( a^{(0)}_{ij})$, $ A^{(1)}=( a^{(1)}_{ij}),~  A^{(2)}=( a^{(2)}_{ij}),~ A^{(3)}=( a^{(3)}_{ij})\in\mathbb{R}^{m\times n} $,  and  $ a^{(1)}_{ij}$, $ a^{(2)}_{ij}$ and $ a^{(3)}_{ij}$ are  respectively the red, green and blue pixel values at the location $(i,j)$ in the image,  $a^{(0)}_{ij}\equiv 0$.
The real counterpart  of $\Aq$ is
\begin{equation*}\label{rc}
  A^{(\mathcal{S})}:=\left[
  \begin{array}{rrrr}
    A^{(0)} & A^{(2)} & A^{(1)} &  A^{(3)} \\
    -A^{(2)} & A^{(0)} &  A^{(3)} & -A^{(1)} \\
    -A^{(1)} & - A^{(3)} & A^{(0)} & A^{(2)} \\
    - A^{(3)} &A^{(1)} & -A^{(2)} & A^{(0)} \\
  \end{array}
\right].
\end{equation*}
Clearly, $A^{(\mathcal{S})}$ is JRS-symmetric. We can apply the  the multi-symplectic Lanczos method on $A^{(\mathcal{S})}$ (see  Section \ref{s:MSL}) to obtain the low-rank approximations of the color image represented by $\Aq$.

In order to get the$k$ largest or smallest singular triplets of the quaternion matrix $\Aq$, we project its real counterpart matrix $A^{(\mathcal{S})}$ into low dimensional space by using multi-symplectic Lanczos bidiagonalization
\begin{equation*}\label{rpld.pa2}
  A^{(\mathcal{S})}P^{(\mathcal{S})}_{k}=Q^{(\mathcal{S})}_{k}B^{(\mathcal{S})}_{k},
\
   (A^{(\mathcal{S})})^{T}Q^{(\mathcal{S})}_{k}=P^{(\mathcal{S})}_{k}(B^{(\mathcal{S})}_{k})^{T}+r^{(\mathcal{S})}_{k}(e_{k}^{(\mathcal{S})})^{T}.
\end{equation*}
 Then we obtain the approximate triplets of $A^{(\mathcal{S})}$ by computing the SVD of $B_{k}$ and applying matrix $P^{(\mathcal{S})}_{k}$ and $Q^{(\mathcal{S})}_{k}$.
Combining (\ref{svd.Bk}) with \eqref{rpld.pa}, we get
\begin{equation}\label{svd.rA2}
   A^{(\mathcal{S})}\tilde{v}_{j}^{(\mathcal{S})}= \tilde{u}_{j}^{(\mathcal{S})}\tilde{\sigma}_{j}^{(\mathcal{S})}, \ (A^{(\mathcal{S})})^{T}\tilde{u}_{j}^{(\mathcal{S})}=\tilde{v}_{j}^{(\mathcal{S})}\tilde{\sigma}_{j}^{(\mathcal{S})}+r_{k}^{(\mathcal{S})}(e_{k}^{(\mathcal{S})})^{T}u_{j}^{(\mathcal{S})}, 1\leq j\leq k.
\end{equation}
The equations in (\ref{svd.rA2}) suggest that an approximate singular triplet $\{{\tilde{\sigma}_{j}^{(\mathcal{S})},\tilde{u}_{j}^{(\mathcal{S})},\tilde{v}_{j}^{(\mathcal{S})}}\}$ is accepted as a singular triplet of $A^{(\mathcal{S})}$ if $r_{k}^{(\mathcal{S})}(e_{k}^{(\mathcal{S})})^{T}u_{j}^{(\mathcal{S})}$ is sufficiently small.
Let $\tilde{\uq}_{j}$ and $\tilde{\vq}_{j}$ be the quaternion vectors with real counterparts $\tilde{u}_{j}^{(\mathcal{S})}$ and $\tilde{v}_{j}^{(\mathcal{S})}$, then $(\tilde{\sigma}_{j},\tilde{\uq}_{j}, \tilde{\vq}_{j} )$ is an approximated singular triplet of quaternion matrix $\Aq$. 
With denoting  
 $$\Uq_k=[\tilde{\uq}_{1},\cdots,\tilde{\uq}_{k} ],~S_k=\texttt{diag}(\tilde{\sigma}_{1},\cdots,\tilde{\sigma}_{k}), ~\Vq_k=[\tilde{\vq}_{1},\cdots,\tilde{\vq}_{k} ],$$
 we obtain a low-rank approximation of $\Aq$ by 
 \begin{equation}\label{e:Ak}
  \Aq_{k}=\Uq_{k}S_{k}\Vq_{k}^{*}.
\end{equation}

The low-rank approximations, defined by \eqref{e:Ak}, have distinct physical meanings in color image processing. The  approximation,  generated by taking  large  singular triplets,  reflects the low frequency information of the color image; meanwhile, we reconstruct the high  frequency information  by taking  small  singular triplets. For instance, the left graph in Figure \ref{fig:toyexample} is a color image with an additional  Gaussian noise; 
the middle graph in Figure \ref{fig:toyexample} denotes the low frequency information, reconstructed by using $k$ largest  singular triplets;  and the right graph in Figure \ref{fig:toyexample} denotes the high  frequency information, reconstructed by using $k$ smallest  singular triplets. 
With this advantage in mind, we can apply the proposed  multi-symplectic Lanczos method to solve the practical problems  from color image processing, such as color face recognition, color video compressing and reconstruction, and many others.

 \begin{figure}[!h]
  \begin{center}
\includegraphics[height=3.7cm, width=3.6cm]{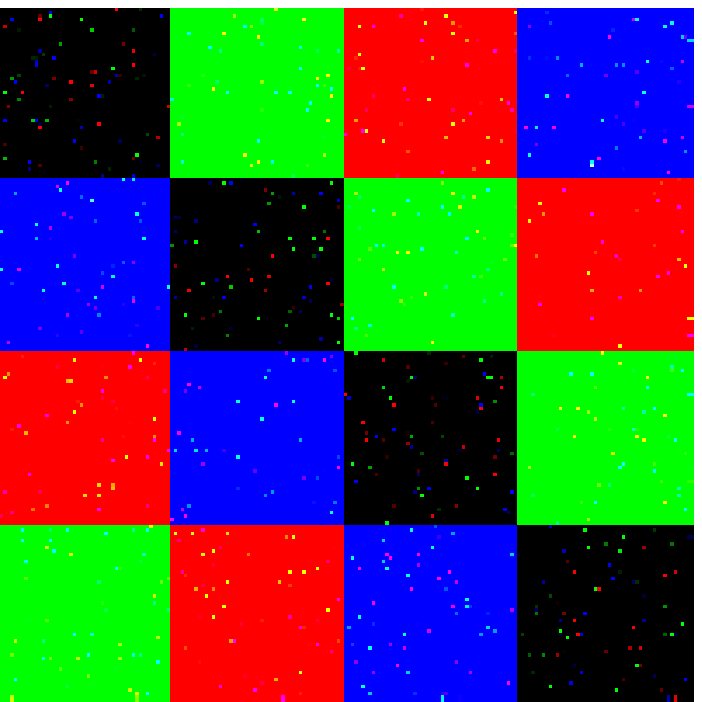}
\includegraphics[height=3.7cm, width=3.6cm]{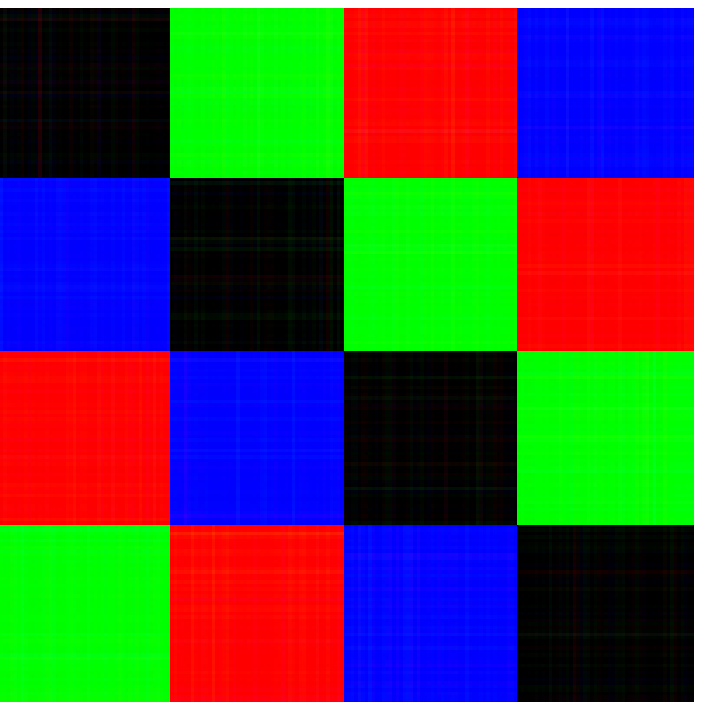}
\includegraphics[height=3.7cm, width=3.6cm]{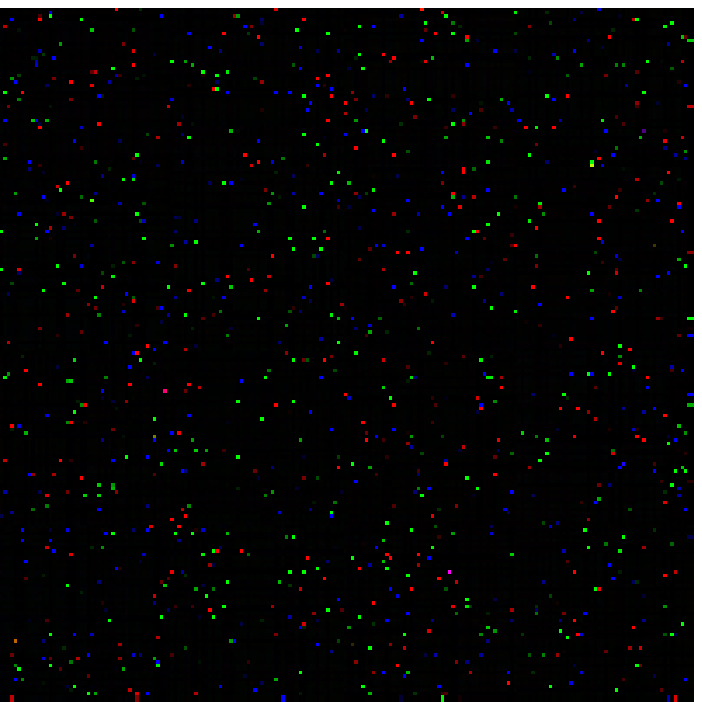}
 \end{center}
 \caption{Low-rank approximations of color image.  A noised color image $\Aq$ of size $200\times 200$ (left), the reconstruction  by using   $k$ largest singular triplets (middle), and the reconstruction by using $k$ smallest singular triplets (right).}\label{fig:toyexample}
\end{figure}

\section{Numerical Experiments}\label{s:ne}
In this section, we compare our algorithms with  five state-of-the-art algorithms  by several numerical examples.  These  algorithms are listed as follows.
\begin{itemize}
\item \texttt{irlba(R)}--implicitly restarted Lanczos bidiagonalization method (Ritz vector) \cite{bare05}.
\item \texttt{irlba(H)}---implicitly restarted Lanczos bidiagonalization method (harmonic vector) \cite{bare05}.
\item \texttt{irlbaMS(R)}--implicitly restarted  multi-symplectic Lanczos bidiagonalization method (Ritz vector) in Section $\ref{s:augritz}$.
\item  \texttt{irlbaMS(H)}--implicitly restarted  multi-symplectic Lanczos bidiagonalization method (harmonic vector) in Section $\ref{augHritz}$.
\item \texttt{lansvdQ}--the Lanczos method for partial quaternion singular value \cite{jns19na}.
\item \texttt{eigQ}--the quaternion eigenvalue decomposition \cite{jwl13}.
\item \texttt{svdQ}--the quaternion singular value decomposition \cite{Li2014,jns19na}.
\end{itemize}   
\bigskip

 All the experiments were performed in Matlab on a personal computer with 3.20
GHz Intel Core i5-3470 processor and 8 GB memory. 
According to equation \eqref{svd.rA}, our numerical method outputs $\{\sigma_{j}^{(\mathcal{S})}, \tilde{u}_{j}^{(\mathcal{S})}, \tilde{v}_{j}^{(\mathcal{S})}\}$ as a singular triplet of $M^{(\mathcal{S})}$ if 
$$
 \beta_{k} |(e_{k}^{(\mathcal{S})})^{T}u_{j}^{(\mathcal{S})}|\leq \delta \|M^{(\mathcal{S})}\|
$$
for a user-specified value of $\delta$, where we have used $\beta_{k}=\|r_{k}^{(S)}\|$. The quantity $\|M^{(\mathcal{S})}\|$  is easily approximated by the singular value $\sigma_{1}^{(B_{k})}$ of largest magnitude of the bidiagonal matrix $B_{k}$. The computation of $\sigma_{1}$ is inexpensive because the matrix $B_{k}$ is small.
The residuals are calculated by
\begin{equation*}
  \text{Residual}=\|M^{(\mathcal{S})} V_k^{(\mathcal{S})}-U_k^{(\mathcal{S})}\Sigma_k^{(\mathcal{S})}\|_F,
\end{equation*}
where $U_k^{(\mathcal{S})}, \Sigma_k^{(\mathcal{S})}$ and $V_k^{(\mathcal{S})}$  are  formed by $k$ calculated singular triplets.  
The  parameters are defined as follows:
\vskip6pt

\begin{tabular}{|l|l|l|}
  \hline
   Notation &Meaning & Default  Value \\ \hline
 $k$ &Number of desired singular triplets & $10$ \\
  $maxit$ & Maximum number of restarts& $2000$ \\
  $\delta$ & Tolerance value& 1.0E-10 \\
  $m_b$ & Size of the Lanczos bidiagonal matrix $B_{k}$ &$\max(2k,40)$\\
  \hline
\end{tabular}

\vskip18pt
\begin{example}[Sparse and JRS-Symmetric Matrix]\label{sparse}
 In this example,  let  the large-scale  matrix $M^{(\mathcal{S})} \in\mathbb{R}^{4n\times 4n}$   be of the form \eqref{e:Mstr} and let  four sparse matrices $M^{(0)}$, $M^{(1)}$, $M^{(2)}$ and $M^{(3)}$  be  the order-$n$ principle submatrices  of  $\texttt{bcspwr10}$, $\texttt{af23560}$, $\texttt{rw5151}$ and  $\texttt{rdb5000}$ (from Matrix Market\footnote{https://math.nist.gov/MatrixMarket}).
 Setting $n=3000$ and $m_b=40$, we compute  the $k$ largest and smallest singular triplets of $M^{(\mathcal{S})}$ by algorithms \texttt{irlba(R)}, \texttt{irlba(H)}, \texttt{irlbaMS}(R) and \texttt{irlbaMS}(H). The CPU times and accuracies are listed in Table \ref{Ex4_max_diff} and \ref{Ex4_min_diff}, respectively. The notation  $n.c.$ means that the algorithm does not converge in $2000$ restarts.

From these numerical results, we can see that the multi-symplectic algorithms, \texttt{irlbaMS(R)} and \texttt{irlbaMS(H)}, are fast and more accuracy than the standard methods,  \texttt{irlba(R)} and \texttt{irlba(H)}.  Obviously, \texttt{irlbaMS(H)} performs better than \texttt{irlbaMS(R)}  on computing the partial smallest singular triplets, while the later  performs better on the calculating the partial  largest singular triplets. 
In the case of computing the $k$ smallest singular triplets,  both \texttt{irlba(R)} and \texttt{irlba(H)} do not converge after $2000$ restarts  for $k=1$ and $5$, and need more than  $1300$ restarts to converge when $k=10$. In Figure \ref{fig:EX2_convergence}, we draw the convergence curves  of the first $10$ singular triplets, i.e., $\beta_{k} |(e_{k}^{(S)})^{T}u_{j}^{(S)}|$. These convergence curves indicate that the multi-symplectic algorithms converge   fast and stably, which indicate the advantages of  the structure-preserving transformations.

\begin{table}[!h]
\centering
\caption{Calculating the first $k$ largest  singular triplets}
\vskip6pt

\begin{tabular}{|c|c|c|c|c|}
  \hline
 $k$ &$Algorithm$  & $Iter$ & $CPU\ time$ & $Residual$  \\
  \hline
                   &$\texttt{irlba(R)}$ &1& 2.0543& 2.5685e-12 \\
                   &$\texttt{irlba(H)}$ & 1& 2.0722 &1.6620e-12 \\
  $k=1$
                   &$\texttt{irlbaMS(R)}$ &2&\textbf{0.5568}& \textbf{5.6299e-13}\\
                   &$\texttt{irlbaMS(H)}$ &2 &0.6155&6.2253e-13 \\
 \hline
                    &$\texttt{irlba(R)}$ &2&3.1450 & 2.6519e-12 \\
                       &$\texttt{irlba(H)}$&2 &3.1055& 2.6195e-12 \\
  $k=5$
                      &$\texttt{irlbaMS(R)}$ &5& \textbf{1.0656} & \textbf{2.4038e-12}\\
                     &$\texttt{irlbaMS(H)}$ &5&1.3251& 2.8619e-12\\
  \hline

                      &$\texttt{irlba(R)}$ &3&  4.5286& 2.6184e-12\\
                       &$\texttt{irlba(H)}$ &3&4.4223&2.4965e-12 \\
  $k=10$
                      &$\texttt{irlbaMS(R)}$ &8&\textbf{1.6678} & \textbf{2.4815e-12}\\
                     &$\texttt{irlbaMS(H)}$ &8&1.7613& 4.8367e-12\\
  \hline
               &$\texttt{irlba(R)}$ &13& 5.5461& 1.5729e-08 \\
               &$\texttt{irlba(H)}$ &12& 5.6529& 1.6638e-08\\
  $k=20$
               &$\texttt{irlbaMS(R)}$ &28& \textbf{2.3643} & \textbf{7.9533e-12}\\
               &$\texttt{irlbaMS(H)}$ &27& 2.5107 & 1.3867e-11\\
  \hline
\end{tabular}

\label{Ex4_max_diff}
\end{table}

\begin{table}[!h]
\centering
\caption{Calculating the first $k$ smallest singular triplets}

\vskip6pt

\begin{tabular}{|c|c|c|c|c|}
  \hline
 $k$ &$Algorithm$  & $Iter$ & $CPU\ time$ & $Residual$  \\
  \hline
                   &$\texttt{irlba(R)}$ &n.c.&- & - \\
                   &$\texttt{irlba(H)}$ & n.c.& - &-\\
  $k=1$
                   &$\texttt{irlbaMS(R)}$ &101&\textbf{14.7813}&6.9991e-12\\
                   &$\texttt{irlbaMS(H)}$ &101&15.9442&\textbf{2.7430e-12} \\
 \hline
                    &$\texttt{irlba(R)}$ &n.c.&-& - \\
                       &$\texttt{irlba(H)}$&n.c. &-& - \\
  $k=5$
                      &$\texttt{irlbaMS(R)}$ &56&\textbf{ 8.4573} & 4.1946e-12\\
                     &$\texttt{irlbaMS(H)}$ &57& 8.7713& \textbf{3.7882e-12} \\
  \hline

                    &$\texttt{irlba(R)}$ &1867&2.2191e+03&2.7272e-02\\
                   &$\texttt{irlba(H)}$ &1357&1.6563e+03&  1.5109e-07  \\
  $k=10$
                      &$\texttt{irlbaMS(R)}$ &50 &\textbf{6.4715} & 8.1583e-12\\
                     &$\texttt{irlbaMS(H)}$ &48&6.5680& \textbf{3.5577e-12}\\
  \hline
               &$\texttt{irlba(R)}$ &814& 2.9945e+02& 1.1936e+01 \\
               &$\texttt{irlba(H)}$ &929&3.5441e+02& 6.3242e-05\\
  $k=20$
               &$\texttt{irlbaMS(R)}$ &154&\textbf{10.6774}& 6.4510e-11\\
               &$\texttt{irlbaMS(H)}$ &145&11.3145 &\textbf{4.6488e-11}\\
  \hline
\end{tabular}

\label{Ex4_min_diff}
\end{table}

 \begin{figure}[!h]
  \begin{center}
\includegraphics[height=5cm, width=6cm]{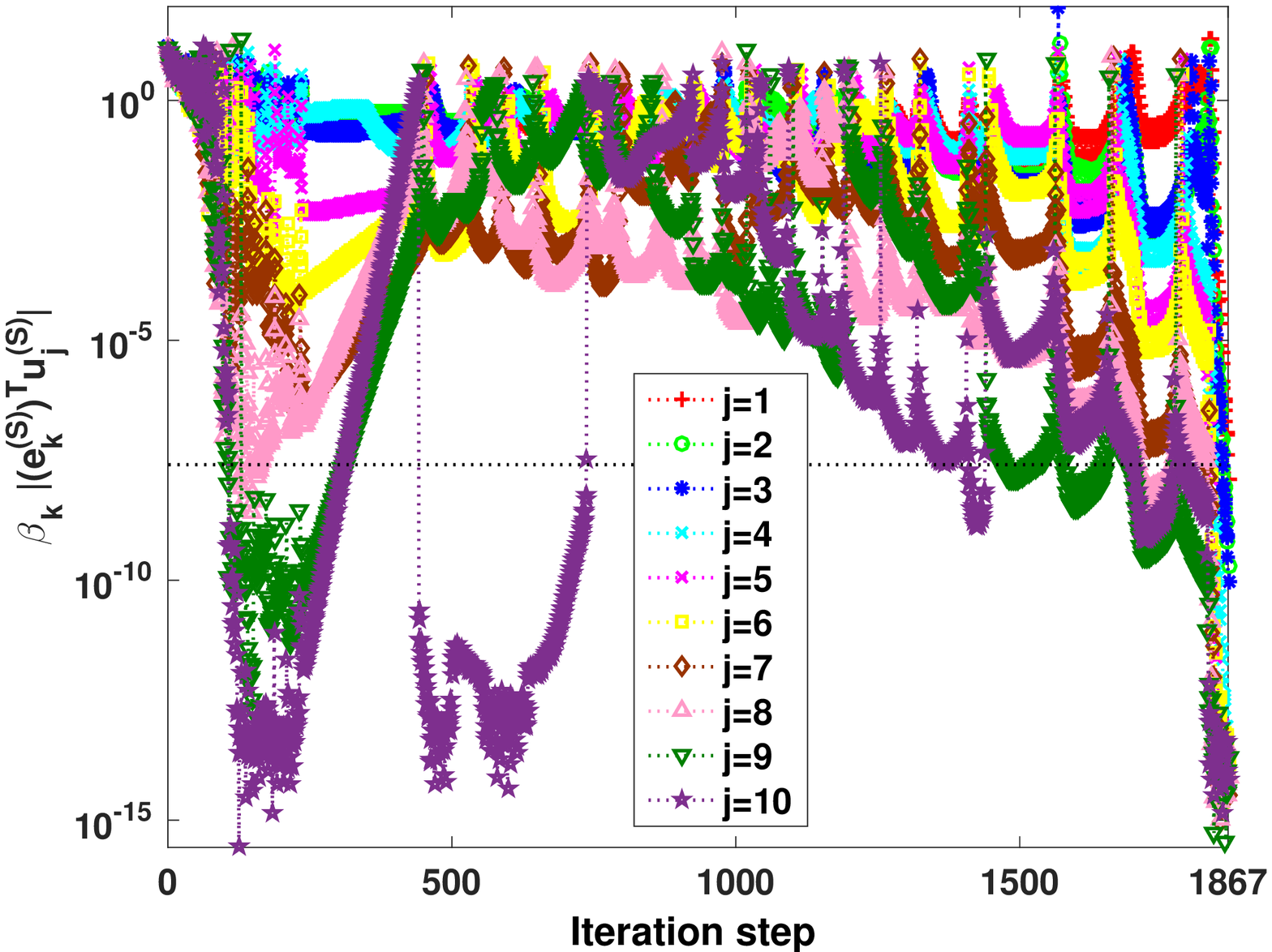}
\includegraphics[height=5cm, width=6cm]{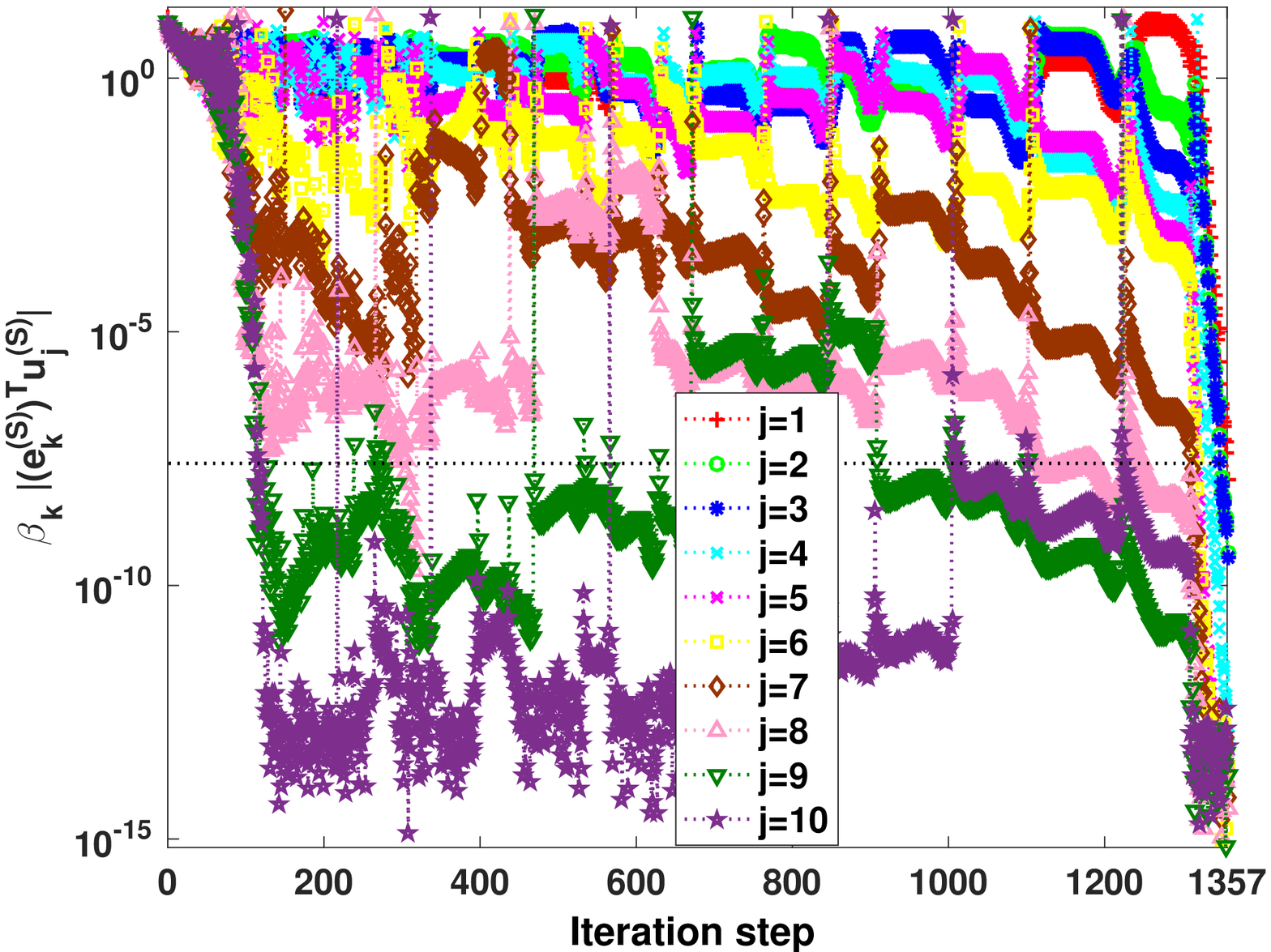}
\includegraphics[height=5cm, width=6cm]{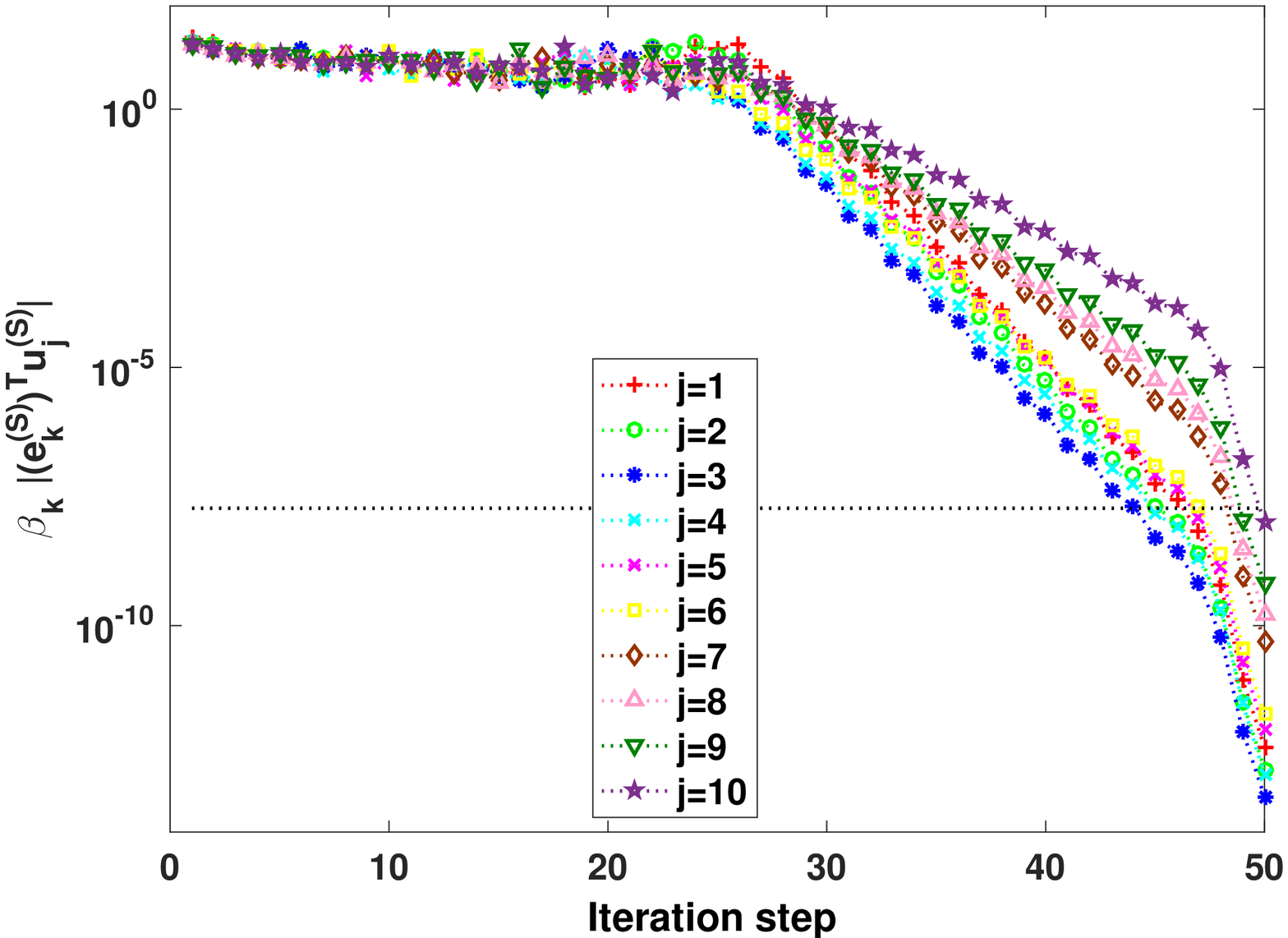}
\includegraphics[height=5cm, width=6cm]{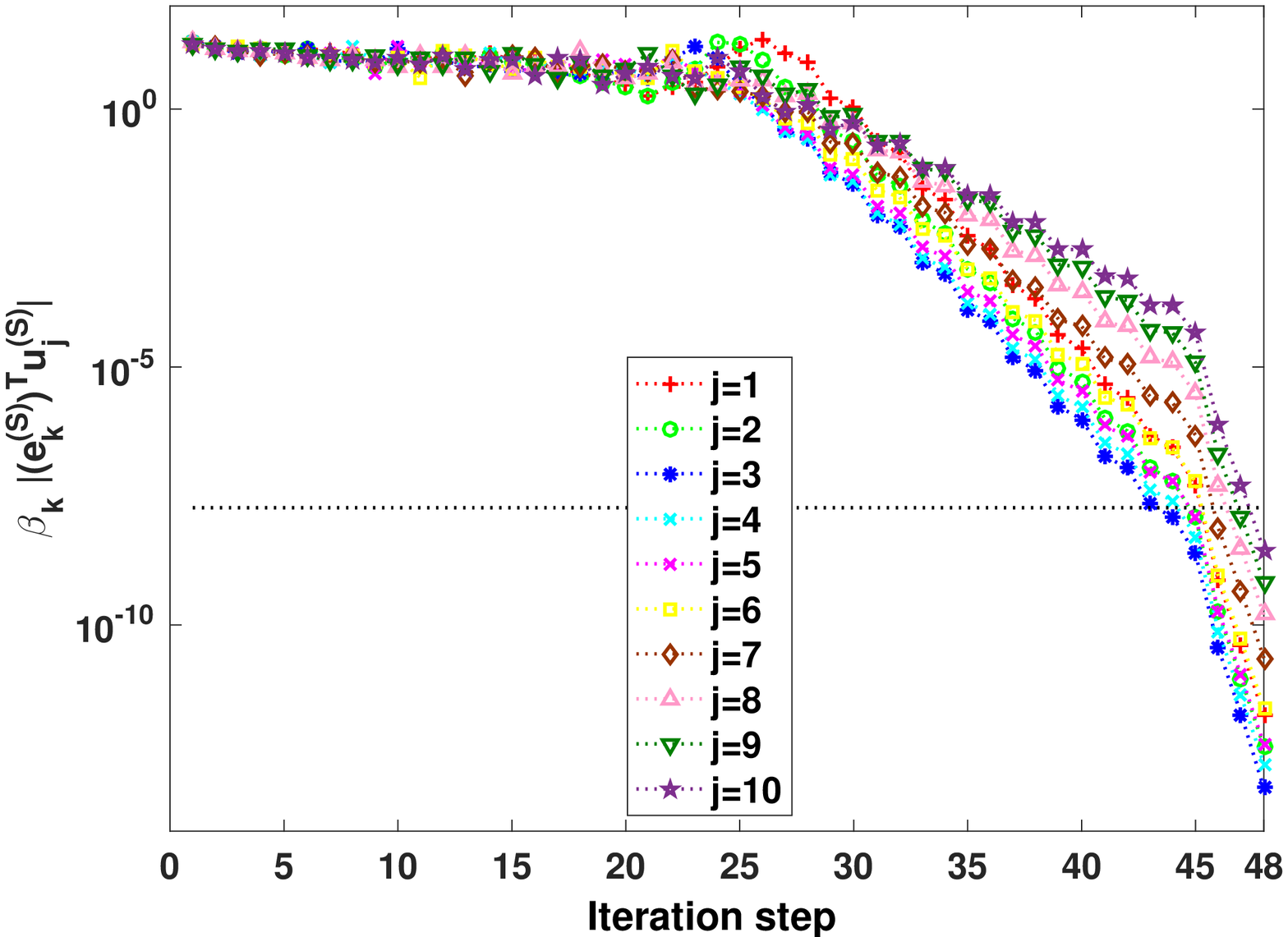}
 \end{center}
 \caption{The convergence curves of the first $10$ largest singular triplets by   \texttt{irlba(R)} (top left), \texttt{irlba(H)} (top right) , \texttt{irlbaMS}(R) (bottom left) and \texttt{irlbaMS}(H) (bottom right)}\label{fig:EX2_convergence}
\end{figure}

\end{example}

\vskip18pt
\begin{example}[Dense Quaternion Matrices]
In this example, the quaternion matrix $\Aq\in \mathbb{Q}^{m\times n}$ is randomly generated as in \cite{jns19na}. 
The algorithms \texttt{irlbaMS(R)}, \texttt{irlbaMS(H)}, \texttt{irlba(R)}, \texttt{irlba(H)}, \texttt{svdQ}, and \texttt{lansvdQ} are applied to compute the $k$ largest or smallest singular triplets of $\Aq$. The parameters  are set as follows. $m=2000,~n=200:200:2000$ and $m_b=40$.

In Figure \ref{fig:EX3_large_cpu2residual2},  we compute the first $k$ ($=1$ or $10$) largest  singular triplets of $\Aq$. In both cases,  \texttt{irlbaMS(R)} and \texttt{irlbaMS(H)} are  faster  than  \texttt{irlba(R)},  \texttt{irlba(H)} and \texttt{svdQ};  \texttt{lansvdQ} is the fastest one.  
When $k=1$, all five algorithms have comparable   residual norms; when $k$ increases, say $k=10$,  the residues of \texttt{lansvdQ} are larger than  those of other methods. 
In Figure \ref{fig:EX3_small_cpu2residual2}, we compute the first $k$ ($=1$ or $10$) smallest singular triplets.  \texttt{irlbaMS(R)} and \texttt{irlbaMS(H)} are indicated to be  faster and more stable  than \texttt{irlba(R)},  \texttt{irlba(H)} and \texttt{svdQ} in general.    
  We can conclude from Figure \ref{fig:EX3_large_cpu2residual2} and Figure \ref{fig:EX3_small_cpu2residual2} that  the multi-symplectic methods, \texttt{irlbaMS(R)} and \texttt{irlbaMS(H)},  perform better than the traditional methods,  \texttt{irlba(R)} and \texttt{irlba(H)}. 

 \begin{figure}[!h]
  \begin{center}
\includegraphics[height=5cm, width=6cm]{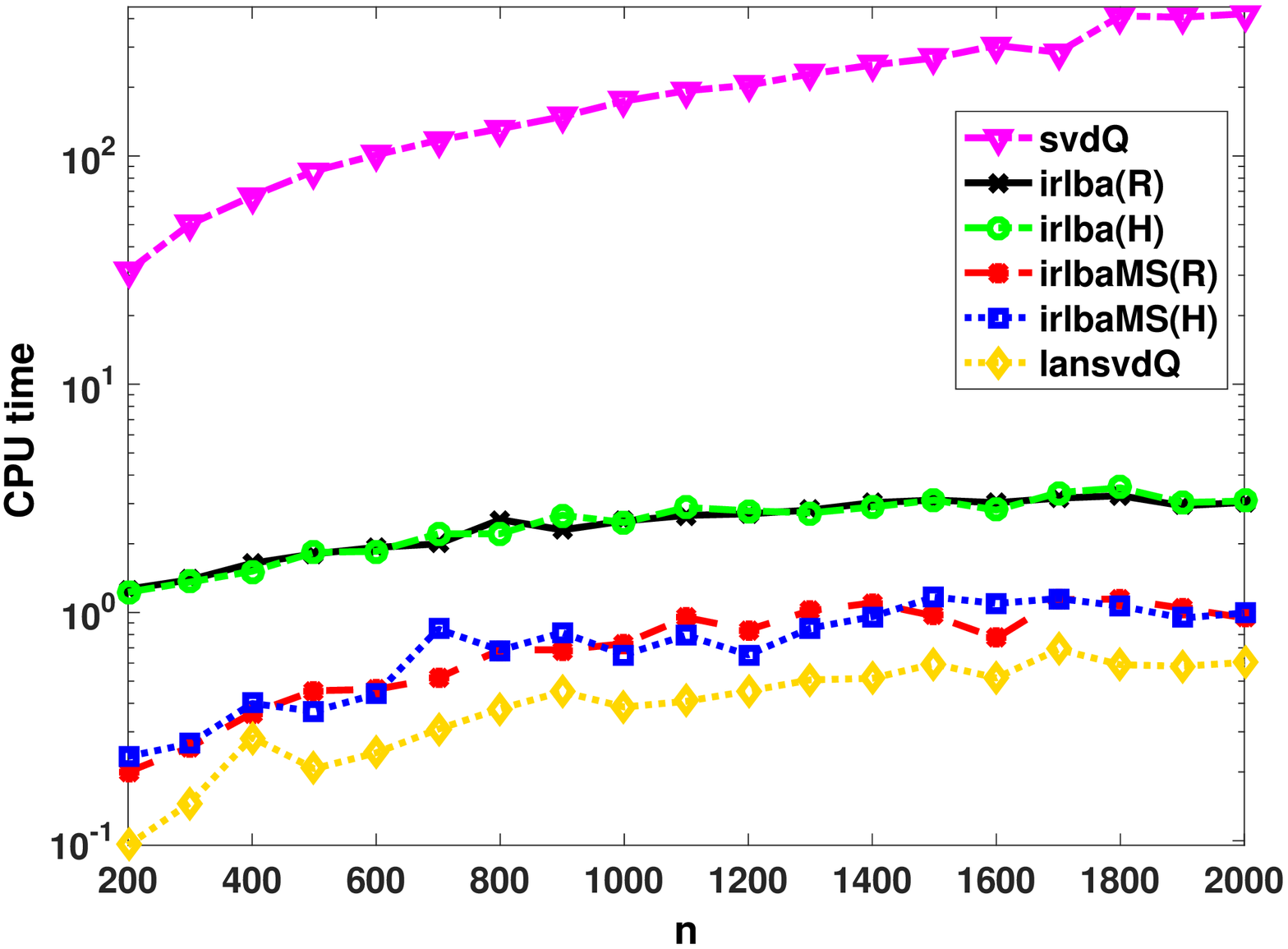}
\includegraphics[height=5cm, width=6cm]{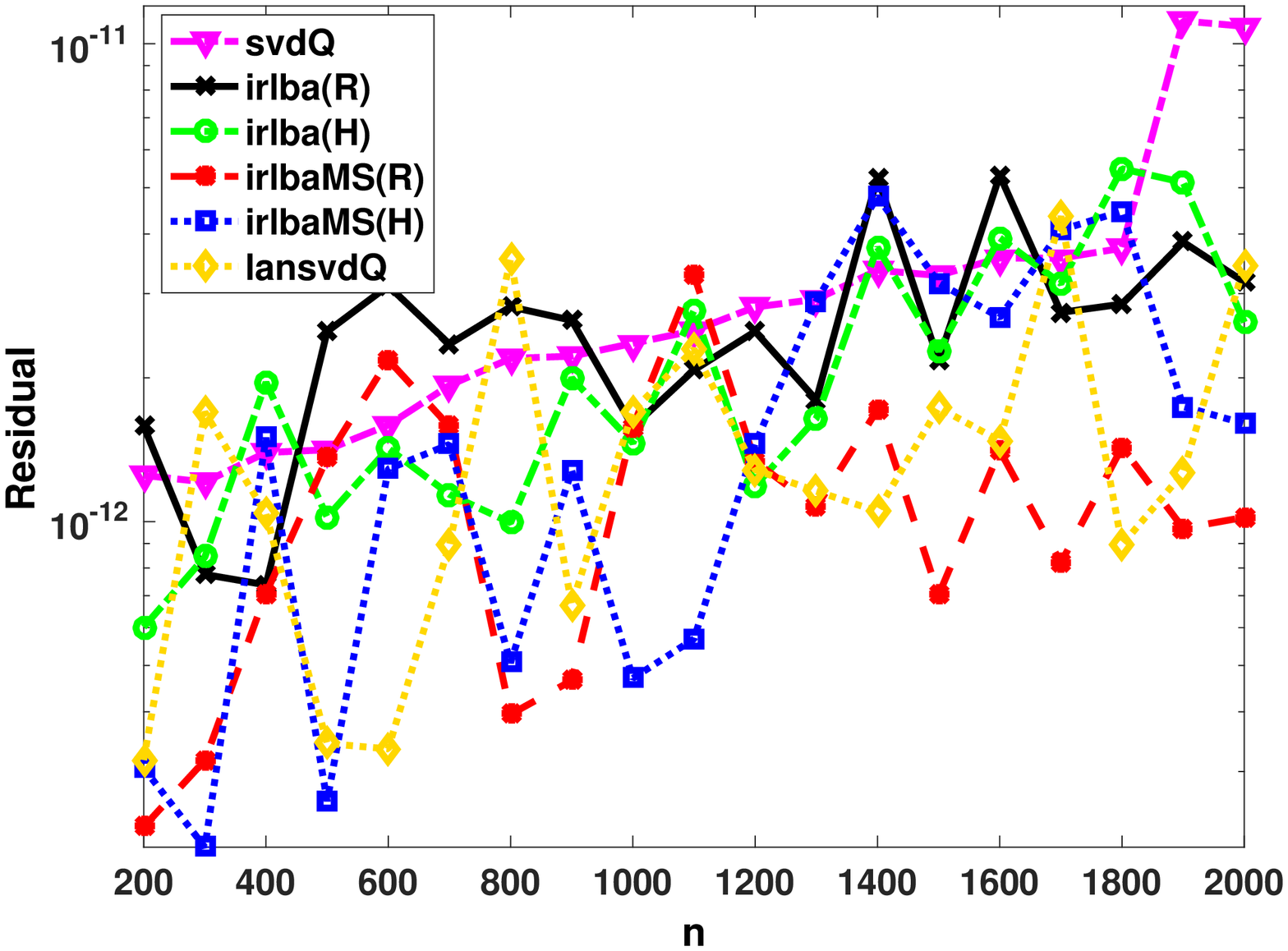}
\includegraphics[height=5cm, width=6cm]{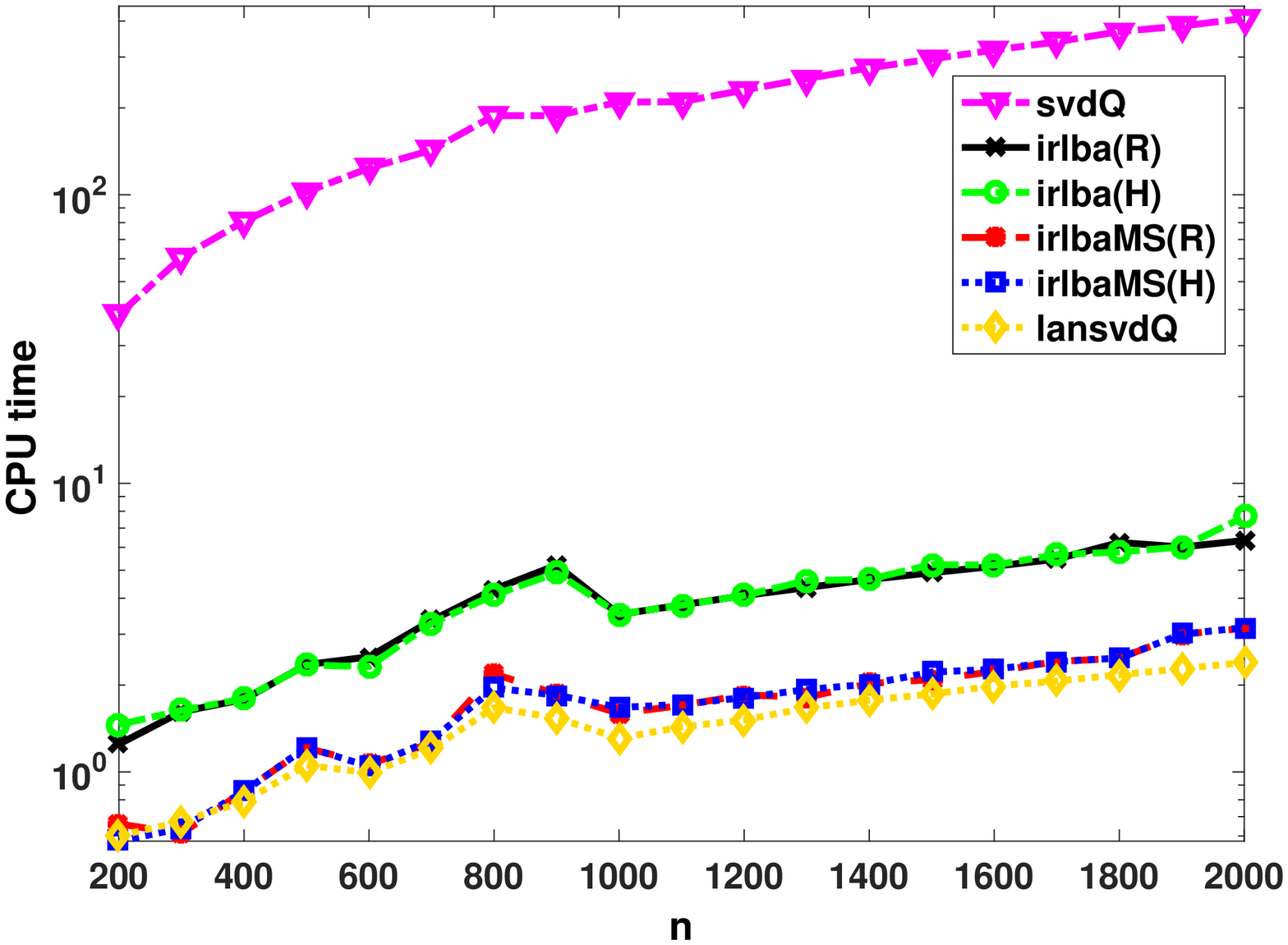}
\includegraphics[height=5cm, width=6cm]{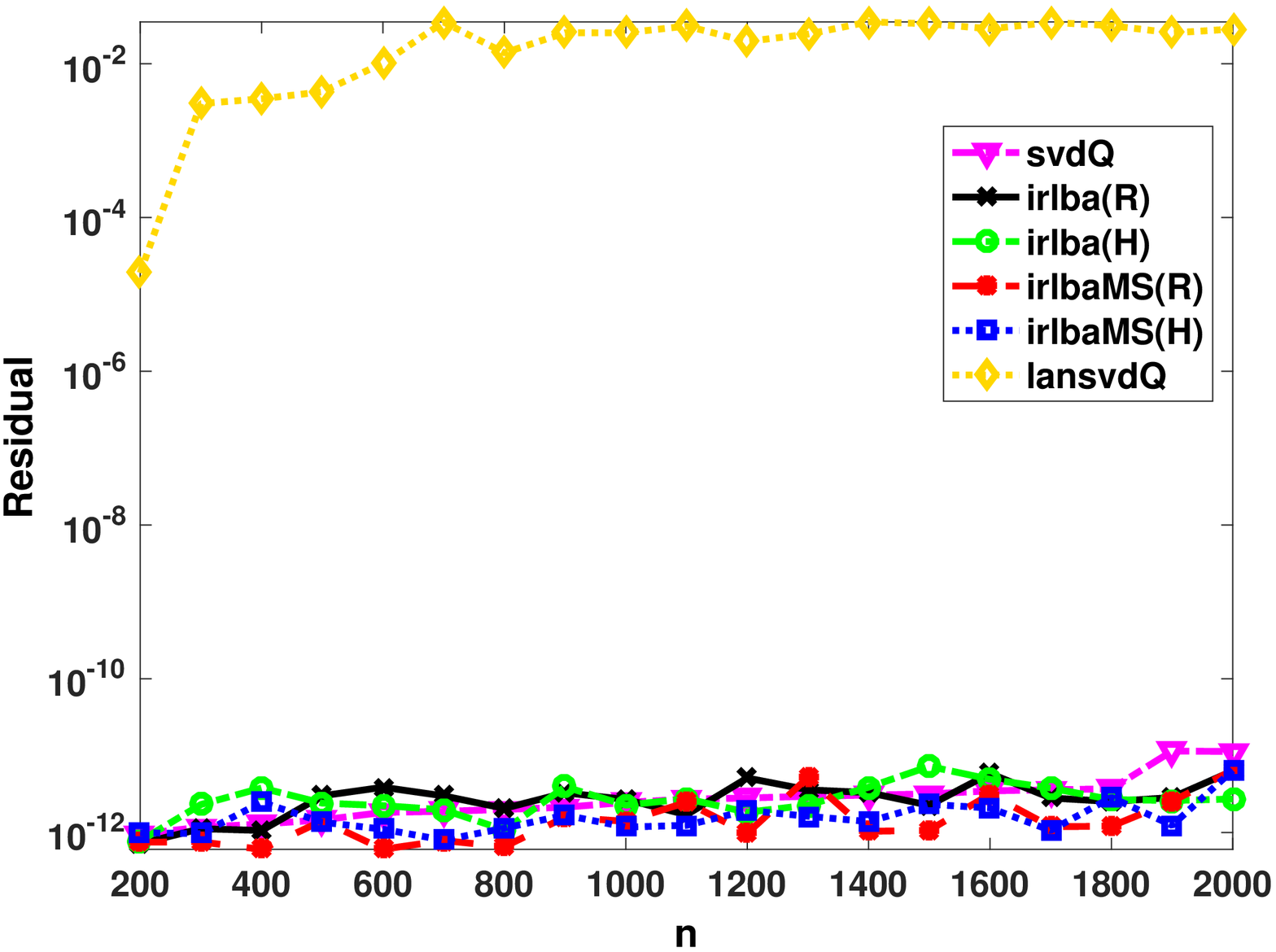}
 \end{center}
 \caption{CPU times (left column) and residuals (right column) on calculating  the first $k$ largest singular triplets. Parameters setting: $m_b=40$ and $k=1$(top row) or $10$ (bottom row).}\label{fig:EX3_large_cpu2residual2}
\end{figure}

 \begin{figure}[!h]
\begin{center}
\includegraphics[height=5cm, width=6cm]{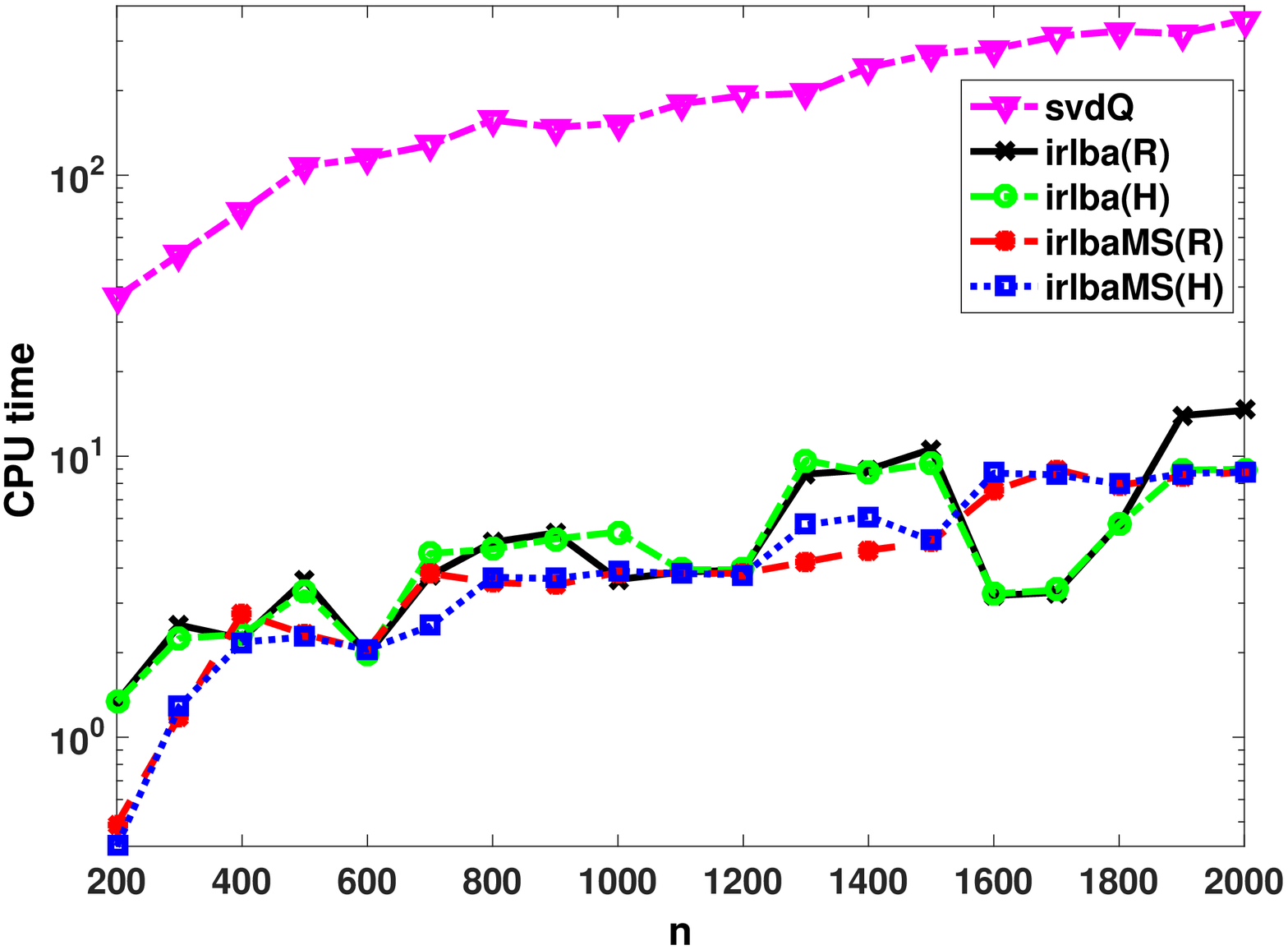}
\includegraphics[height=5cm, width=6cm]{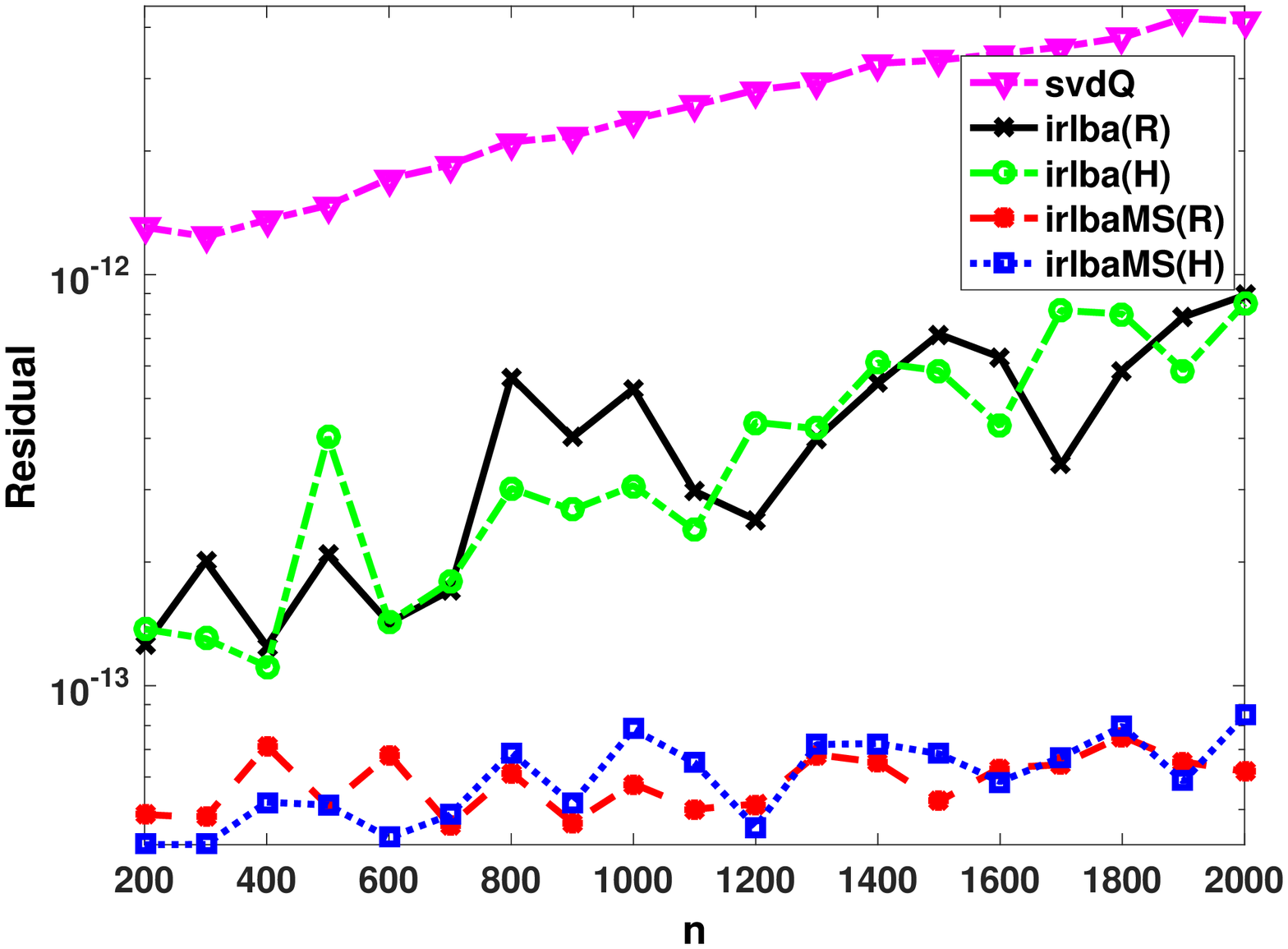}
\includegraphics[height=5cm, width=6cm]{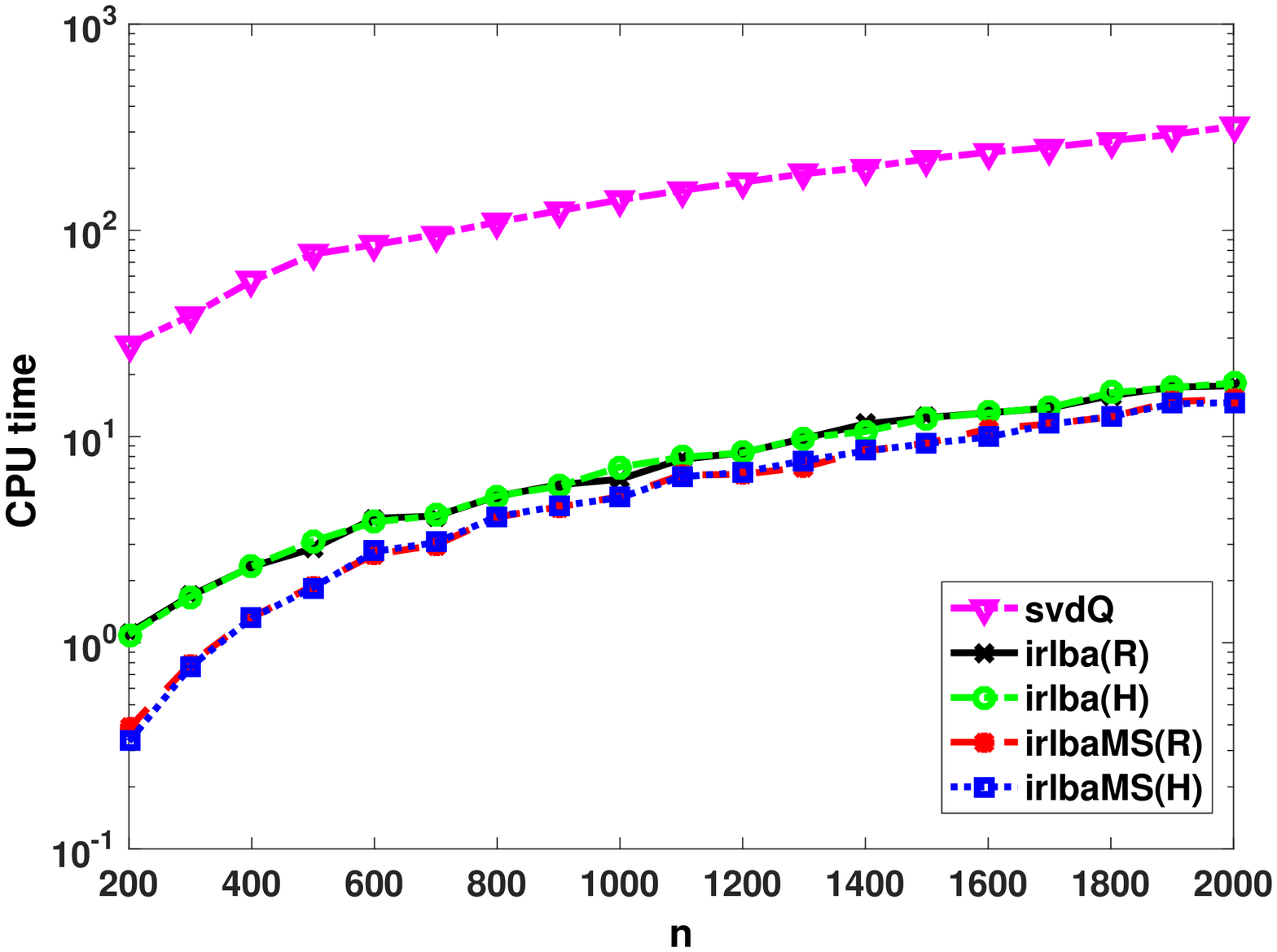}
\includegraphics[height=5cm, width=6cm]{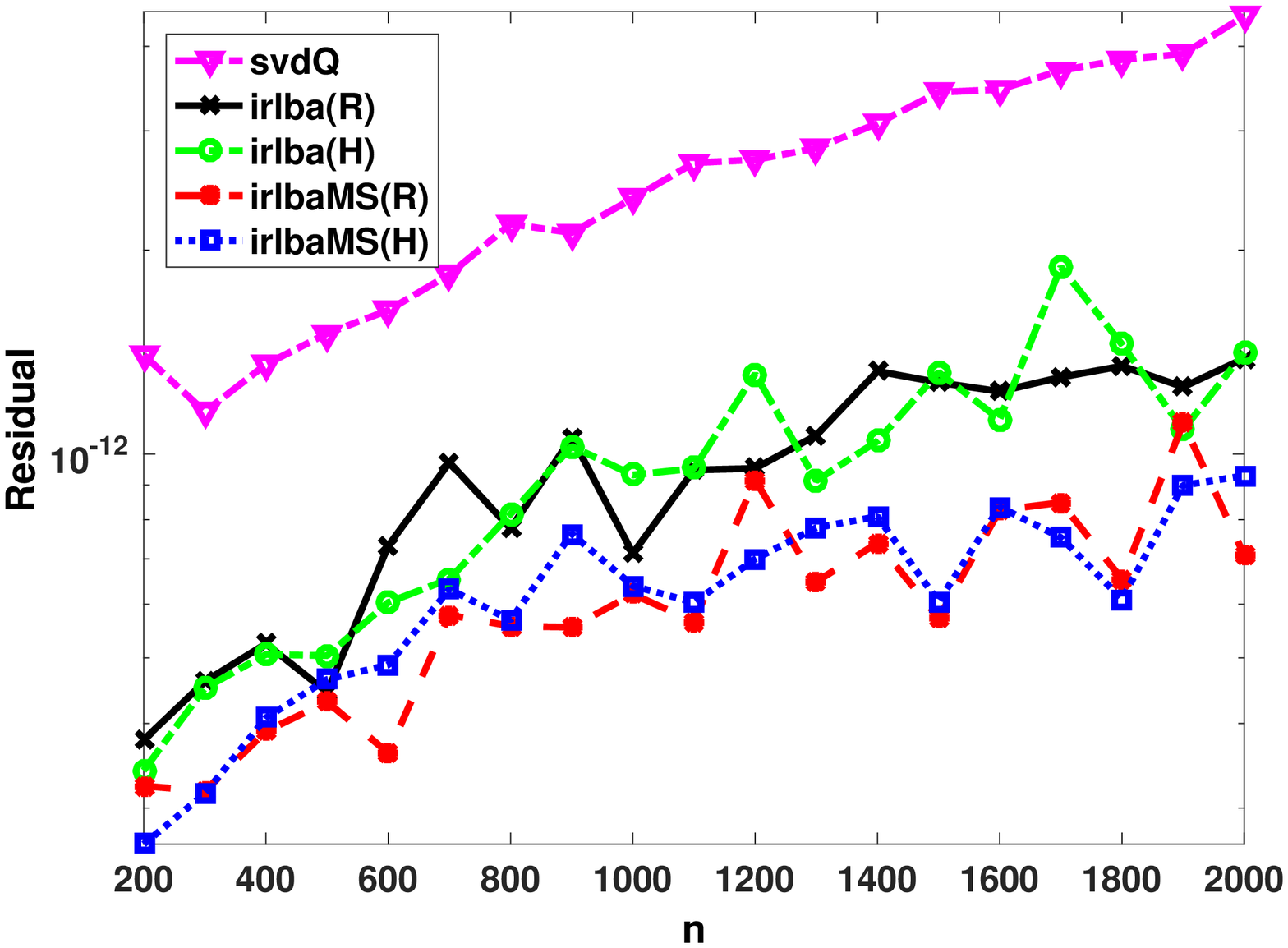}
\end{center}
 \caption{CPU times (left column) and residuals (right column) on calculating  the first $k$ smallest singular triplets. Parameters setting: $m_b=40$ and $k=1$(top row) or $10$ (bottom row).}\label{fig:EX3_small_cpu2residual2}
\end{figure}

\end{example}

\vskip18pt
\begin{example}[Color Face Recognition]\label{ex:cfr}
Color information is one of the most important characteristics in reflecting the structural information of an image. Face recognition performance with color images can be significantly better than that with grey-scale images; see, e.g., \cite{jlz17}.
Suppose that there are $l$ training color image samples,
denoted by $m \times n$ pure quaternion matrices $\Fq_{1}, \Fq_{2},\ldots, \Fq_{l}$, and the average is $\Psi=\frac{1}{l}\sum_{s=1}^{l}\Fq_{s}\in \mathbb{Q}^{m\times m}$. Let
\begin{equation*}
  \Xq=[vec(\Fq_{1})-vec(\Psi),\cdots,vec(\Fq_{l})-vec(\Psi) ],
\end{equation*}
where $vec(\cdot)$ means to stack the columns of a matrix into a single long vector. The
core work of color principal component analysis is to compute the right
singular vectors corresponding to first $k$ largest singular values of $\Xq$. These vectors
are called eigenfaces, and generate the projection subspace, denoted as $\Vq$.

In this experiment, we apply \texttt{lansvdQ, eigQ and irlbaMS(R)}  into color
images principal component analysis, based on the  Georgia Tech face database\footnote{The Georgia Tech face database. http://www.anefian.com/research/facereco.htm}. The setting is same to  \cite{jns19na}.
All images in the Georgia
Tech face database are manually cropped, and then resized to $120\times 120$ pixels. There
are $50$ persons to be used. The first ten face images per individual person are chosen
for training and the remaining five face images are used for testing. The number of
chosen eigenfaces, $k$, increases from 1 to 30. In each case, computing $k$ largest
singular triplets of a  quaternion matrix, $\Xq$, of size $14400\times 500$. Here the $14400$ rows refer to
$120\times120$ pixels and the $500$ columns refer to $50$ persons with $10$ faces each.

The detailed comparison on CPU times and recognition accuracies
of these methods is shown in Figure \ref{fig:EX8}. We see  that \texttt{irlbaMS(R)}
is  faster than other the two algorithms, while the recognition accuracies  are almost
same.
\begin{figure}[!t]
\center\includegraphics[height=5cm, width=12cm]{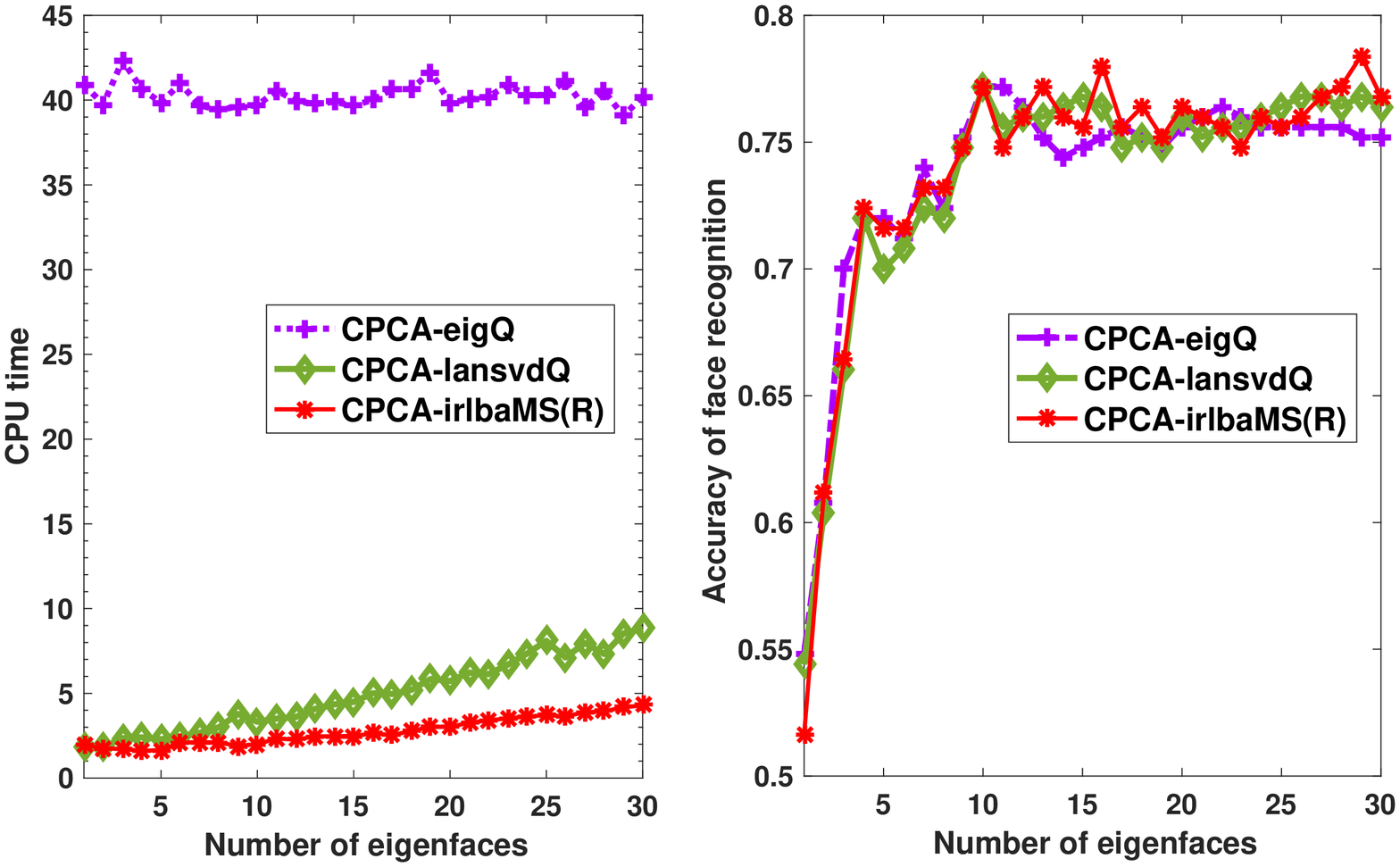}
  \caption{The CPU time and the accuracy by \texttt{lansvdQ, eigQ and irlbaMS(R)}.}\label{fig:EX8}
\end{figure}
\end{example}

\begin{example}[Color Video Compressing and Reconstruction]\label{ex:cv}
All frames of
a color video can be stacked into a quaternion matrix, $\Aq\in \mathbb{Q}^{(lm)\times n}$, where $l$ is
the number of frames, and $m$ and $n$ denote numbers of rows and columns of each
frame, respectively. 
Based on the SVD theory of quaternion matrix, the optimal rank-$k$ approximation to $\Aq$ can be reconstructed from its first $k$ largest singular values and corresponding left/right singular vectors. Denote
such approximation as
\begin{equation*}
  \Aq_{k}=\Uq_{k}S_{k}\Vq_{k}^{*},
\end{equation*}
where $S_{k} = \texttt{diag}(\sigma_{1},\ldots,\sigma_{k})$ consists of the first $k$ largest singular values of $\Aq$, and
the columns of $\Uq_{k}$ and $\Vq_{k}$ are left and right corresponding singular vectors. Then the relative distances from $\Aq_{k}$ to $\Aq$ are
\begin{equation}\label{AK-A}
  \frac{\|\Aq_{k}-\Aq\|_{2}}{\|\Aq\|_{2}}=\frac{\sigma_{k+1}}{\|\Aq\|_{2}},\ \ \frac{\|\Aq_{k}-\Aq\|_{F}}{\|\Aq\|_{F}}=\frac{(\sum_{j=k+1}^{min(lm,n)}\sigma_{j}^{2})^{\frac{1}{2}}}{\|\Aq\|_{F}}.
\end{equation}

In this example, we apply \texttt{lansvdQ},  \texttt{eigQ} and \texttt{irlbaMS(R)} to compute $k$ largest quaternion singular triplets of a
large-scale video quaternion matrix.  
The  direct method \texttt{eigQ}   is firstly used to compute the eigenvalue problem of the Hermitian quaternion matrix  $\Aq^{*}\Aq$.  The eigenvectors of $\Aq^{*}\Aq$ are in fact the  left singular vectors of $\Aq$, saved in a unitary matrix $\Vq$, and the square roots of the eigenvalues are the singular values, saved in a
diagonal matrix $S$. The right singular vectors of $\Aq$ is generated by $\Uq = \Aq\Vq S^{-1}$. Let
$\hat{\Uq}_{k}=\Uq(:,1:k)$, $\hat{S}_{k}=S(1 : k, 1 : k)$, and $\hat{\Vq}_{k}=\Vq(:, 1 : k)$, then we get another
approximation to $\Aq$, $\hat{\Aq}_{k}=\hat{\Uq}_{k}\hat{S}_{k}\hat{\Vq}_{k}$. We expect that these methods can achieve at
the same accuracy, but the costed  CPU time of \texttt{irlbaMS(R)} is shortest. 

Let $\Fq$ and $\Fq_{k}$ denote the original frame and its approximation, respectively. The
peak signal-to-noise ratio (PSNR) value of $F_{k}$ is defined as
\begin{equation*}
  PSNR(F_{k},F)=10*log_{10}(\frac{255^{2}mn}{\|F_{k}-F\|_{F}^{2}})
\end{equation*}
The structural similarity (SSIM) index of $F_{k}$ and F is defined as
\begin{equation*}
  SSIM(F_{k},F)=\frac{(4\mu_{x}\mu_{y}+c_{1})(2\sigma_{xy}+c_{2})}{(\mu_{x}^{2}+\mu_{y}^{2}+c_{1})(\sigma_{x}^{2}+\sigma_{y}^{2}+c_{2})},
\end{equation*}
where $x$ and $y$ denote the vector forms of $F$ and $F_{k}$, respectively, $\mu_{x,y}$ denotes the
average of $x$, $y$, $\sigma_{x,y} $ the variance of $x$, $y$, $\sigma_{xy}$ the covariance of $x$ and $y$, and $c_{1,2}$ are
two constants.

First, we take the color video $yunlonglake.mp4$ provided in \cite{jns19na} for testing the efficiencies
 of \texttt{lansvdQ, eigQ and irlbaMS(R)} on color video compressing and reconstruction. This video consists of 31 frames (each frame is of size 544$\times$960). Second, we test two methods with the color video, \texttt{children.mov}(from \cite{jns19na}), which consists of 20 frames and each
frame is of size 1280$\times$360.  The relative distances defined as in \eqref{AK-A}, PSNR and
SSIM values of the approximations to randomly chosen 10 frames by three methods
are computed and the average values are shown in Table \ref{table:EX9}. The CPU times of
computing $k$ singular triplets by \texttt{lansvdQ, eigQ and irlbaMS(R)} are shown in Figure \ref{fig:demo4}.
We can see that  \texttt{irlbaMS(R)} is the fastest one.

\begin{table}[!t]
\centering
\caption{The average relative distances, PSNR and SSIM values of the approximations of 10 frames with applying $k$ singular triplets ($k = 30$).}
\vskip  15pt
\begin{tabular}{|c|c|c|c|c|c|}
  \hline
  Video & Method & PSNR& SSIM &  \large{$\frac{\|A_{k}-A\|_{2}}{\|A\|_{2}}$} & \large{$\frac{\|A_{k}-A\|_{F}}{\|A\|_{F}}$} \\
  \hline
             & \texttt{eigQ} & 29.5953  &  0.9236 &0.0148  &   0.0620\\
  yunlonglake & \texttt{lansvdQ} &  29.5953 & 0.9236  &  0.0148 &     0.0620\\
            & \texttt{irlbaMS(R)} & 29.5953 & 0.9236  & 0.0148   &    0.0620 \\
  \hline
    & \texttt{eigQ} &  23.4508 & 0.8428  & 0.0273  &  0.1243 \\
  children & \texttt{lansvdQ}& 23.4508  &  0.8428 & 0.0273  & 0.1243   \\
    & \texttt{irlbaMS(R)} & 23.4508& 0.8428  &  0.0273&  0.1243   \\
  \hline
\end{tabular}
\label{table:EX9}
\end{table}

\begin{figure}[!t]
\center \includegraphics[height=5cm,width=12cm]{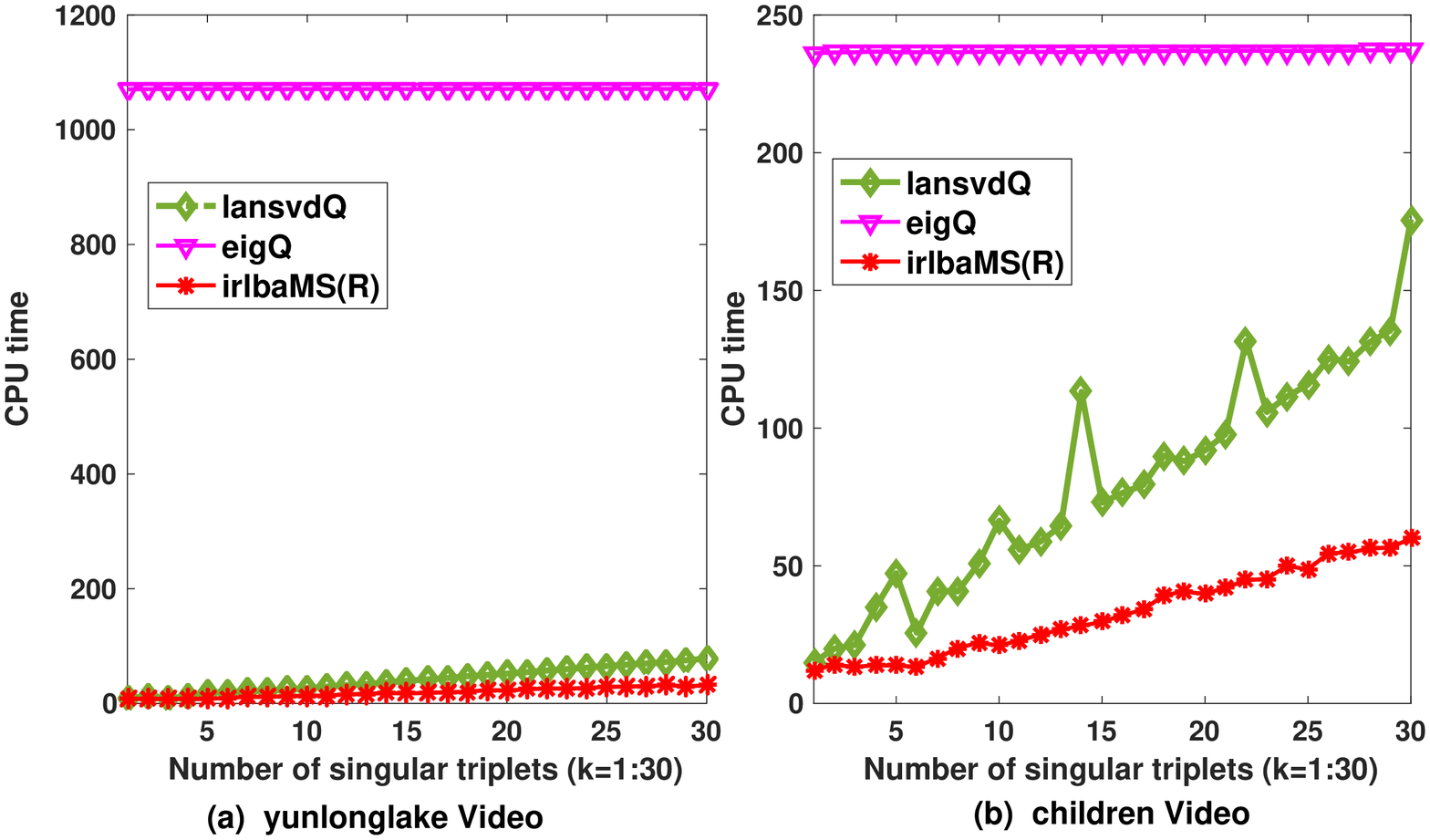}
  \caption{The computational CPU times of $k$ singular triplets by \texttt{lansvdQ, eigQ and irlbaMS(R)}.
x-axis: the number of singular triplets; y-axis: CPU times.}\label{fig:demo4}
\end{figure}

\end{example}

\section{Conclusion}\label{s:clu}
In this paper, a multi-symplectic Lanzcos  method is proposed and applied to compute the $k$ largest and smallest singular triplets of JRS-symmetric matrix. 
The augmented Ritz  and  harmonic Ritz vector are applied respectively to perform implicitly restarting to obtain a satisfactory bidiagonal matrix. Two new associated algorithms are presented: \texttt{irlbaMS(R)}    and \texttt{irlbaMS(H)}.  \texttt{irlbaMS(H)} is the first reliable algorithm of computing $k$ smallest singular triplets to the best of our knowledge. 
The proposed multi-symplectic algorithms  performs better than the state-of-the-art algorithms in  numerical experiments.

\section*{Acknowledgments}
We are grateful to Prof. Michael K. Ng from The University of Hong Kong for his helpful suggestions to the  primary version of this paper.

\bibliographystyle{siamplain}

\end{document}